\documentclass[12pt]{amsart}
\usepackage{amssymb,amsmath,amsthm,mathrsfs}
\usepackage{eulervm}
\usepackage[dvipsnames]{xcolor}
\usepackage{graphicx} 
\usepackage[linktocpage=true,backref=page]{hyperref}
\hypersetup{
	colorlinks=true, 	
	linktoc=all,     		
	linkcolor=black,  	
	citecolor=Emerald,
}
\calclayout
\newtheorem{theorem}{Theorem}[section]
\numberwithin{theorem}{section} \numberwithin{equation}{section}
\newtheorem{corollary}[theorem]{Corollary}
\newtheorem{proposition}[theorem]{Proposition}
\newtheorem{definition}[theorem]{Definition}

\newtheorem{lemma}[theorem]{Lemma}
\newcommand\scalemath[2]{\scalebox{#1}{\mbox{\ensuremath{\displaystyle #2}}}}
\title[Horizontal 2-plane distributions on 5-manifolds]{Geometric aspects of rank-3 vector bundles over surfaces and 2-plane distributions on 5-manifolds}
\author{Brandon P. Ashley}
\address{Dept.\!~of Mathematics, Southern Oregon University, Ashland, OR 97520}
\email{ashleyb@sou.edu}
\author{Michael T.  Schultz}
\address{Dept.\!~of Mathematics, Virginia Tech, Blacksburg, VA 24060}
\email{michaelschultz@vt.edu}
\begin{document}
	\begin{abstract}
		We study geometric aspects of horizontal 2-plane distributions on the complement of the zero section in the 5-dimensional total space of a rank-3 vector bundle equipped with connection over a surface. We show that any surface in 3-dimensional projective space can be associated to such a geometric structure in 5-dimensions, and establish a dictionary between the projective differential geometry of the surface and the growth vector of the 2-plane distribution.
	\end{abstract}
	\maketitle
	\tableofcontents
	\section{Introduction}
	\label{s-intro}
	\par In differential geometry, the role of a connection on the tangent bundle of a manifold is of central importance to the intrinsic geometry of the space. Gauge-theoretic considerations ask the question of how more general geometric structures may be associated to a given manifold. One interaction between these two areas can be seen in the induced extrinsic geometry obtained by immersing the manifold into an ambient space and restricting the ambient geometry onto the image. 
	\par The classical differential geometry of surfaces in $3$-dimensional euclidean space $\mathbb{R}^3$ illustrates this interaction. Let $\Sigma$ be a smooth surface and suppose that $\psi : \Sigma \to \mathbb{R}^3$ is an immersion. Then there is a natural orthogonal splitting of the rank-3 vector bundle $B = \psi^* T\mathbb{R}^3 \to \Sigma$ as $B = T\Sigma \oplus N_{\Sigma}$ induced by the first fundamental form of $\Sigma$ imparted by the immersion $\psi$. Here $N_{\Sigma}$ is the normal bundle of $\Sigma \subset \mathbb{R}^3$, which we identify with its image $\psi(\Sigma)$ by working locally, and may regard $N_{\Sigma} \to \Sigma$ and $B \to \Sigma$ as being a trivial vector bundles. The fibres of $B$ carry the obvious $\mathrm{GL}(3)$ representation, and we may view $B$ as being an associated bundle of a left principal $\mathrm{GL}(3)$-bundle over $\Sigma$.
	\par The first fundamental form induces a canonical symmetric metric connection on $T\Sigma$, the Levi-Civita connection, and the normal bundle $N_{\Sigma}$ inherits a connection determined by the second fundamental form. These two connections fit together to form a $\mathrm{GL}(3)$-connection $\nabla$ on the rank-3 bundle $B$ as follows. Let $(x,y)$ be local coordinates on $\Sigma$. Then the immersion $\psi$ can be locally represented by 
	\begin{equation}
		\label{eq-euclidean_immersion}
		(x,y) \mapsto z(x,y) = \left(z_1(x,y),z_2(x,y),z_3(x,y)\right) \in \mathbb{R}^3 \, .
	\end{equation}
	Then $p = z_x$ and $q = z_y$ locally span the tangent bundle $T\Sigma$, and we may complete to a frame $e = (\sigma,p,q)$ of $B$ by choosing an appropriate section $\sigma$ of the normal bundle $N_{\Sigma}$ with unit norm. It is well known that $z$ satisfies the second order system of differential equations
	\begin{equation}
		\label{eq-euclidean_equations}
		\begin{aligned}
			& z_{x x}=\Gamma^1_{11} z_x+\Gamma^2_{11} z_y+l \sigma \, , \\
			& z_{x y}=\Gamma^1_{12} z_x+\Gamma^2_{12} z_y+m \sigma \, , \\
			& z_{y y}=\Gamma^1_{22} z_x+\Gamma^2_{22} z_y+n \sigma \, ,
		\end{aligned}
	\end{equation}
	where $\Gamma^i_{jk}$ are the Christoffel symbols defining the Levi-Civita connection, and $l,m,n$ are the coefficients of the second fundamental form. Equation (\ref{eq-euclidean_equations}), known classically as the Gauss formula, yields a $\mathfrak{gl}(3)$-valued 1-form $\omega$ on $\Sigma$ given by 
	\begin{equation}
		\label{eq-euclidean_connection}
		\omega = \left(\begin{array}{ccc}
			0 & dx & d y \\
			l \, d x + m \, dy & \Gamma^1_{11}\, dx + \Gamma^1_{12} \, dy & \Gamma^2_{11} \, d x + \Gamma^{2}_{12} \, dy  \\
			m \, d x + n \, dy &  \Gamma^1_{12}\, dx + \Gamma^1_{22} \, dy & \Gamma^2_{12} \, d x + \Gamma^{2}_{22} \, dy \\
		\end{array}\right) \,,
	\end{equation}
	encoding the connection 1-forms of $\nabla$.
	\par This connection naturally contains information on the geometry of $\Sigma \subset \mathbb{R}^3$. For example, a simple calculation of the curvature tensor $\mathbf{R} = d\omega - \omega \wedge \omega$ reveals that the classical Gauss-Codazzi equations
	\begin{equation}
		\label{eq-codazzi}
		\begin{aligned}
			& m_x - l_y = l\Gamma^1_{12} + m(\Gamma^2_{12}-\Gamma^1_{11}) - n\Gamma^2_{11} \, , \\
			& n_x - m_y = l\Gamma^1_{22} +m(\Gamma^2_{22}-\Gamma^1_{12}) - n\Gamma^2_{12} \, , 
		\end{aligned}
	\end{equation}
	are simply the vanishing of certain components of the curvature tensor; it is a basic result that $\mathbf{R} = 0$ if $\Sigma \subset \mathbb{R}^3$ is a plane. One may obtain further identities by utilizing the differential Bianchi identity $d_{\nabla}\mathbf{R}=0$, where $d_{\nabla}$ is the exterior covariant derivative. On the other hand, it is useful to insist that the frame $e$ be orthonormal in $B$, which is equivalent to a reduction of the structure group of $B$ from $\mathrm{GL}(3)$ to $\mathrm{SO}(3)$. This leads to the well-known Darboux frame away from umbilic points of $\Sigma$ (see e.g., \cite[Ch. 2]{Ivey_Landsberg2003}), which leads to an analogous connection that is canonically flat.
	\par Other geometric information can be seen by examining covariant derivatives of the curvature $\mathbf{R}$ with respect to $\nabla$. In this case, a coarse invariant --- the growth vector --- can be obtained by studying the derived flag of the exterior differential system (EDS) of the connection 1-forms pulled back to complement of the zero section in the 5-dimensional total space $\mathrm{Tot}(B)$ of the vector bundle $B$. Let $M^5 = \mathsf{Z}_B^c \subset \mathrm{Tot}(B)$, where $\mathsf{Z}_B \subset \mathrm{Tot}(B)$ is the zero section of the bundle $\pi : B \to \Sigma$. Then it is easy to see that the differential ideal $\Omega :=\langle\pi^*\omega\vert_{M^5}\rangle_{\mathrm{diff}}$ generated by the three 1-forms on $M^5$,
	\begin{equation}
		\label{eq-euclidean_1forms}
		\begin{cases}
			\omega_{\sigma} := d\sigma - pdx - qdy \, , \\
			\omega_x := dp - (l\sigma + p\Gamma^1_{11}+q\Gamma^2_{11})dx - (m\sigma +p\Gamma^1_{12}+q\Gamma^2_{12})dy \, , \\
			\omega_y := dq - (m\sigma + p\Gamma^1_{12}+q\Gamma^2_{12})dx - (n\sigma +p\Gamma^1_{22}+q\Gamma^2_{22})dy \, ,
		\end{cases}
	\end{equation}
	determines a rank-3 EDS $\Omega = \langle \omega_\sigma, \omega_x, \omega_y \rangle_{\mathrm{diff}}$ on $M^5$ from the connection 1-forms on $\Sigma$. Dualizing, $\Omega$ determines a rank-2 distribution $\Delta = \Omega^\perp \subset TM^5$, i.e., $\Delta$ is defined as the rank-2 sub-bundle of $TM^5$ such that $\Omega = \mathrm{ann}(\Delta)$ is the annihilator of $\Delta$. The vector fields $X,Y$ generating $\Delta$ are easily computed as
	\begin{equation}
		\label{eq-euclidean_distribution}
		\begin{cases}
			X = \partial_x + p\partial_{\sigma} + (l\sigma + p\Gamma^1_{11}+q\Gamma^2_{11})\partial_p + (m\sigma +p\Gamma^1_{12}+q\Gamma^2_{12})\partial_q \, , \\
			Y = \partial_y + q\partial_{\sigma} + (m\sigma + p\Gamma^1_{12}+q\Gamma^2_{12})\partial_p + (n\sigma +p\Gamma^1_{22}+q\Gamma^2_{22})\partial_q \, .
		\end{cases}
	\end{equation}
	Generically, the distribution is \emph{bracket generating}, meaning that the derived flag of $\Delta$ --- the collection of all sub-bundles of $TM^5$ generated by iterated Lie brackets of $\Delta = \mathrm{span}\,\{X,Y\}$ --- generates all of $TM^5$. 
	\par In this generic case\footnote{There is another bracket generating case, that of the \emph{contact distribution}, which has growth vector $(2,3,4,5)$. This however, is not generic behavior.}, Cartan showed \cite{Cartan_1910} that such distributions have growth vector $(2,3,5)$, meaning that $\mathrm{rank}(\Delta^\prime)=\mathrm{rank}([\Delta,\Delta]) = 3$ and $\mathrm{rank}(\Delta^{(2)})=\mathrm{rank}([\Delta^\prime,\Delta^\prime])=5$. It is a basic consequence of the connection $\nabla$ that there is a splitting $TM^5 = H \oplus V$ into horizontal and vertical distributions, where the rank-2 distribution $H = \pi^*T\Sigma\vert_{M^5} \cong \Delta$ and $V$ is the rank-3 distribution tangent to the fibres of $B$ along $M^5$. Then the curvature $\mathbf{R}$ of $\nabla$ vanishes if and only if $\Delta$ is integrable in the sense of Frobenius, that is $\Delta^\prime = \Delta$. In the case of nonvanishing curvature, we symbolically have $\Delta^\prime \subset V$ generated by $\mathbf{R}$ and $\Delta^{(2)} \subset V$ generated by the covariant derivative $\nabla \mathbf{R}$.
	\par Cartan demonstrated in loc. cit. that the local equivalence problem associated to generic bracket generating 2-plane distributions on a given $M^5$ exhibits a surprising connection to the (split-real form of the) simple Lie algebra $\mathfrak{g}_2$, which arises as the maximal local symmetry algebra of the distribution. Implicit in Cartan's construction, and later elucidated by Nurowski \cite{MR2157414}, is the existence of a conformal structure of split signature $(3,2)$ on $TM^5$ in which $\Delta$ is totally null, and whose Weyl tensor determines the local obstructions to maximal $\mathfrak{g}_2$ symmetry. The conformally flat case is necessary and sufficient for maximal symmetry. The geometry of such distributions continue to be of interest into modernity, studied in a variety of contexts, for example, in works respectively by Nurowski, C\v ap \& Sagerschnig, Willse, and Randall \cite{MR2384729,MR2536853,MR3159952,MR4612637,MR4520032,MR4895591}. They have arisen in connection to the geometry of surfaces rolling without slipping or twisting, studied in addition to those above for example by Bryant \& Hsu, Zelenko, and Erickson \cite{MR1240644,Zelenko_2006,erickson_surfaces2025}, in connection to the geometry of twistor distributions on split signature semi-riemannian 4-manifolds by An \& Nurowski and Nurowski, Sagerschnig \& The \cite{Nurowski_2014,MR4887495}, and in connection to Darboux integrability of certain nonlinear PDE in the plane, for example by Stormark, Strazzullo, and the dissertation of the first named author \cite{Stormark_2000,MR2717829,MR4380607}. Other recent work by Dave \& Haller, Haller \cite{MR3934472,MR4458384,MR4959549,haller2024,haller2025} has focused on global aspects of $(2,3,5)$-distributions on both open and closed 5-manifolds, with the construction of Dave \& Haller in \cite{MR3934472} producing such distributions on certain closed 5-manifolds associated to a surface $\Sigma$, while requiring that the Euler characteristic of a closed surface $\Sigma$ be $\chi(\Sigma) \geq 0$.
	\par In this article we initiate a study of relationships between two subjects shaped by the work of E. Cartan: on the one hand, his work \cite{Cartan_1920} from 1920 on the differential geometry of surfaces in 3-dimensional projective space --- in particular, those that admit projective deformations; on the other hand, his work \cite{Cartan_1910} from 1910 on 2-plane distributions in 5-dimensions. To set the stage for our work, we utilize in \S \ref{s-rank_3} the well-known horizontal and vertical distributions on the total space of a vector bundle with connection over a manifold. In the case of a rank-3 vector bundle with connection over a surface, we obtain in Lemma \ref{lemma-horizontal_distribution_M5} this horizontal distribution as a 2-plane distribution on the complement of the zero section in the 5-dimensional total space of the bundle. Although elementary, these horizontal distributions appear to have not been widely studied in the relevant literature on 2-plane distributions in 5-dimensions. We prove an easy result (\ref{thm-general_growth_vector}) characterizing the growth vector of such distributions in terms of the curvature of the connection and its covariant derivatives that is useful in our work on projective surfaces.
	\par The connection between the two areas influenced by Cartan comes through local study of an overdetermined rank-4 linear system of PDEs two variables, known to many of the old masters at the turn of the 20th century (Equations (\ref{eq-rank_four_system}, \ref{eq-canonical_sys})). We review this material in \S \ref{s-DG}. We show that for a very general surface $\Sigma \subset \mathbb{P}^3$, the differential invariants of $\Sigma$ that are discussed in detail in \S \ref{s-DG} determine a bracket generating 2-plane distribution $\bar{\Delta}$ on a certain 5-manifold that is the complement of the zero section in the total space of a rank-3 vector bundle equipped with connection over $\Sigma$. The mechanism by which we obtain this geometric data from the rank-4 linear system is a non-traditional form of Cauchy reduction that we term \emph{derived reduction} (Lemma \ref{lemma-cauchy_reduction}). This process may be of use in studying other linear systems of PDE related to projective hypersurfaces \cite{MR996019}, so we take care to develop this idea in \S \ref{s-derived_reduction}. Although the subject of surfaces in $\mathbb{P}^3$ is  classical, our insights here appear to be new, and are provided by analyzing the linear systems of PDE from the perspective of exterior differential systems. As an application, one of our main results is Theorem \ref{thm-proj_surface_growth_vector}, in which we establish a dictionary between the growth vector of the derived flag of the 2-plane distribution $\Delta$ and the projective differential invariants of $\Sigma$. The dictionary resembles the rough classification of projective surfaces into quadrics, ruled surfaces, and general surfaces. We demonstrate the utility of our perspectives in Proposition \ref{prop-3_param_family} by using this correspondence to produce an interesting family of projective surfaces, which we believe to be new.
	\par Part of the novelty of the work in this article is that we study 2-plane distributions on 5-manifolds that are explicitly dependent on the \emph{extrinsic} geometry of a surface, as opposed to much of the existing work in the literature, e.g., in the work referenced above on surfaces rolling without twisting or slipping. Moreover, in most work, it is required that the generic 2-plane distribution be bracket generating, and hence has growth vector $(2,3,5)$. Here, we demonstrate that the geometry of horizontal 2-plane distributions are connected to interesting and somewhat subtle geometry connected to projective surfaces even when they \emph{fail} to be bracket generating.
	\par A brief remark on conventions and notation: the scope of this article includes both the real $C^\infty$ category and complex analytic category, though it is necessary to impose a mild assumption in the former that will be described in \S \ref{ss-rank-4_systems}. We let $\mathbb{K}$ denote either $\mathbb{R}$ or $\mathbb{C}$ as we consider working in either respective category. As such, all relevant geometric data is assumed to be smooth in the appropriate category even if it is not said explicitly. Over $\mathbb{C}$ we take smooth to mean holomorphic and a diffeomorphism to be a biholomorphism, and we always work in the analytic topology. Since we are concerned with projective geometry, we utilize notation for line bundles and sheaves on projective space familiar to algebraic and complex geometers; i.e., $\mathcal{O}(-1) \to \mathbb{P}^n$ is the tautological sheaf and $\mathcal{O}(m) = \mathcal{O}(-1)^{\otimes m}$ for $m \in \mathbb{Z}$, and $\mathcal{O}_{\Sigma}(m) = \mathcal{O}(m) \vert_{\Sigma}$ for $\Sigma \subset \mathbb{P}^n$. If $E \to \Sigma$ is a smooth vector bundle, we write $\underline{E}$ to denote the sheaf of germs of local smooth sections. The superscript $T$ indicates transpose, so that for example if $v \in \underline{E}$, we have $v^T \in \underline{E}^*$ (in some local trivialization), where $E^* \to \Sigma$ is the dual bundle. We write $\mathrm{End}(E) = E \otimes E^*$ as the endomorphism bundle of $E$. For a field $\mathbb{K}$, the symbol $\mathbb{\underline{K}}^n$ denotes the locally constant sheaf whose stalk is isomorphic to $\mathbb{K}^n$. When a local framing of a bundle has been specified, we use the summation convention on repeated indices. We use standard notation for intrinsic vector bundles such as $T\Sigma, \, T^*\Sigma$ to respectively denote the tangent and cotangent bundles of $\Sigma$. The canonical sheaf $K_\Sigma = \wedge^2 \underline{T^*\Sigma}$ is the sheaf of sections of the highest exterior power of the cotangent bundle. $S^p(\Sigma)$, $\Omega^p(\Sigma)$ denotes the respective sheaves of smooth $\mathcal{O}_{\Sigma}$-valued symmetric covariant $p$-tensors and $p$-forms on $\Sigma$, etc., where $\mathcal{O}_{\Sigma}=\Omega^0(\Sigma)$ denote the structure sheaf of local smooth functions on $\Sigma$. Finally $\mathcal{O}^*_{\Sigma} \subset \mathcal{O}_{\Sigma}$ denotes the subsheaf of nonvanishing smooth functions. 
	\par In this context, the results we prove in \S \ref{s-2_plane_distr} apply to a large swath of surfaces of interest. For $\mathbb{K} = \mathbb{C}$, we may associate by our framework a 2-plane distribution to any smooth algebraic surface in $\mathbb{P}^3$. It follows from Theorem \ref{thm-proj_surface_growth_vector} that the distribution is $(2,3,5)$ for general surfaces of nonnegative Kodaira dimension. For $\mathbb{K} = \mathbb{R}$, our framework applies to closed surfaces that may be embedded in real projective 3-space, classified by the following theorem of Bredon \& Wood \cite{MR246312}.
	\begin{theorem}
		\label{thm-bredon_wood}
		A closed surface $\Sigma$ embeds in real projective 3-space if and only if $\Sigma$ is orientable or of odd Euler characteristic.
	\end{theorem}
	See also the short note by Creighton \cite{MR515780} for a geometric proof of this result. It follows in particular that our Theorem \ref{thm-proj_surface_growth_vector} applies to closed, orientable surfaces of negative Euler characteristic $\chi(\Sigma) \leq -2$.
	\begin{flushleft}
		\emph{Acknowledgment.} After an initial draft of this article was posted, we received valuable feedback from Stefan Haller. We thank him for his comments and his references to relevant literature.
	\end{flushleft}
	\section{Rank-3 bundles over a surface}
	\label{s-rank_3}
	Let $\pi : B \to \Sigma$ be a smooth rank-3 vector bundle over a surface $\Sigma$, and $\mathsf{Z}_B \subset~ \mathrm{Tot}(B)$ be the zero section. Assume that $B$ is an associated bundle of a smooth, left principal $\mathrm{G}$-bundle over $\Sigma$ for some Lie group $\mathrm{G}$. Write $\mathfrak{g}$ for the Lie algebra of $\mathrm{G}$. Assume further that $B$ is endowed with a $\mathrm{G}$-connection $\nabla$. Working locally, let $\omega = (\omega^\alpha_\beta) \in \mathrm{\underline{End}}(B) \otimes \Omega^1(\Sigma)$, where $\mathrm{End}(B)$ has been equipped with the associated $\mathfrak{g}$-representation. Let $\mathbf{R} \in \mathrm{\underline{End}}(B) \otimes \Omega^2(\Sigma)$ be the curvature tensor. We have then that $\mathbf{R} = d\omega - \omega \wedge \omega$, and that $g \in \mathrm{G}$ acts on $\nabla$ and $\mathbf{R}$ by $\omega \mapsto \tilde{\omega} = g \omega g^{-1} + dg\cdot g^{-1}$, and $\mathbf{R} \mapsto \tilde{\mathbf{R}} = g\mathbf{R}g^{-1}$. Our calculations are purely local, so we may freely assume that $\Sigma$ is connected and $B$ is trivial. Write then that $\underline{B} = \mathrm{span}\{e_1,e_2,e_3\}$ for global independent sections $e_1,e_2,e_3$ of $B$. 
	\par In the context of this article, we are interested in how such extrinsic geometric data attached to $\Sigma$ is encoded in the intrinsic geometry of the 5-manifold $M^5=\mathsf{Z}_B^c \subset \mathrm{Tot}(B)$, which is endowed with the natural local coordinates $(x,y,z,p,q)$. This is most naturally encoded in the language of exterior differential systems, and their dual vector field systems (see for example \cite{BCG3,Ivey_Landsberg2003,Stormark_2000}). We recall the following basic terminology. 
	\par A distribution $\Delta \subseteq TM$ is a smooth sub-bundle of the tangent bundle of a manifold $M$. The dual object of $\Delta$ is an exterior differential system (EDS) $\Omega = \mathrm{ann}(\Delta) \subseteq T^*M$ consisting of the 1-forms that annihilate all the vector fields in $\Delta$, and naturally determines a differential ideal in the exterior algebra of $M$. The derived distribution of $\Delta$ is by definition the distribution 
	\begin{equation*}
		\Delta' = \Delta + [\, \Delta, \, \Delta \,] \, , 
	\end{equation*}
	spanned by a local basis of vector fields for $\Delta$ and their Lie brackets. A distribution is \emph{integrable} if and only if $\Delta = \Delta^\prime$; this is the classical Frobenius theorem. Dualizing, one shows that
	\begin{equation*}
		\Omega^\prime = \mathrm{ann}(\Delta^\prime) = \{ \, \theta \in \underline{T^*M} \, | \, d\theta \equiv 0 \; \mathrm{mod} \, \Omega \, \} \, 
	\end{equation*}
	forms the derived codistribution. By iteratively computing derived distributions, or their duals in the manner described above, one constructs the increasing filtration of distributions, called the derived flag of $\Delta$,
	\begin{equation}
		\Delta = \Delta^{(0)} \subseteq \Delta^{(1)} \subseteq \cdots \subseteq \Delta^{(\infty)} \subseteq TM,
	\end{equation}
	where $\Delta^{(1)} = \Delta'$, $\Delta^{(k+1)} = (\Delta^{(k)})'$, and $\Delta^{(\infty)} = \Delta^{(\ell)}$ where $\Delta^{(\ell)} = (\Delta^{(\ell)})'$. The vector 
	\begin{equation}
		\left(\mathrm{rank} \, \Delta, \, \mathrm{rank} \, \Delta', \, \ldots, \, \mathrm{rank} \, \Delta^{(\infty)}\right)
	\end{equation}
	is called the \emph{growth vector} of $\Delta$, which is a diffeomorphism invariant of $\Delta \subseteq TM$. A distribution is said to be $\ell$-step bracket generating if $\Delta^{(\ell)}=TM$ and $\Delta^{(\ell-1)} \neq TM$. Of course, there is a corresponding decreasing filtration of $T^*M$ by iteratively derived codistributions, whose growth vector is the vector of coranks.
	\subsection{Geometry of horizontal distributions}
	\label{ss-horizontal}
	\par The discussion in \S \ref{s-intro} clearly generalizes to the situation at hand, in that the pullback of the connection forms determine an EDS on $M^5$ whose derived flag is determined by the curvature $\mathbf{R}$ of $\nabla$ and its covariant derivative $\nabla \mathbf{R}$. We state here the general story. In what follows, saying that ``$B$ is a rank-$k$ $\mathrm{G}$-bundle'' means that $B$ is a vector bundle of rank-$k$ whose structure group is $\mathrm{G} \subseteq GL(k)$; hence the Lie algebra $\mathfrak{g} \subseteq \mathfrak{gl}(k)$.
	\begin{proposition}
		\label{prop-total_space_EDS}
		Let $\pi : B \to \Sigma$ be a trivial rank-$k$ $\mathrm{G}$-bundle over an $n$-manifold $\Sigma$, equipped with a compatible linear connection $\nabla$. Let $\mathsf{Z}_B \subset \mathrm{Tot}(B)$ be the zero section. Write $B = \mathrm{span}\{e_\beta\}$ for global independent sections $e_\beta$, $\beta = 1,\dots,k$. In local coordinates $(x^1,\dots,x^n)$ on $\Sigma$, write the connection 1-forms $\omega = (\omega^\alpha_\beta) \in \mathfrak{g} \otimes \Omega^1(\Sigma)$ as $\omega^\alpha_\beta =~ \Gamma^\alpha_{i\beta}dx^i$. Let $M=\mathsf{Z}_B^c$ be the complement of the zero section in the $(n+k)$-dimensional total space of $B$, with local coordinates $(x^i,e_\beta)$. Then the EDS $\Omega := \langle \pi^*\omega \vert_{M} \rangle_{\mathrm{diff}}$ generated by the pullback of the connection forms is a rank-$k$ Pfaffian system generated by the 1-forms
		\begin{equation}
			\label{eq-general_connection_1-forms}
			\omega_\beta := de_\beta - e_\alpha\Gamma^\alpha_{i\beta}dx^i \, , \quad \beta = 1,\dots,k \, 
		\end{equation}
		which are semibasic with respect to $\pi$. Moreover, $\Omega = \mathrm{ann}(\Delta)$ is the annihilator of an $n$-plane distribution $\Delta \subset TM$, naturally realized as the horizontal distribution $H = \pi^*T\Sigma\vert_{M} \cong \Delta$ of the connection $\nabla$. This horizontal distribution is generated by the vector fields 
		\begin{equation}
			\label{eq-general_horizontal_distribution}
			X_i = \partial_{x^i} + \sum_{\beta=1}^{k} e_\alpha \Gamma^\alpha_{i\beta}\partial_{e_\beta} \, , \quad i = 1,\dots,n \, .
		\end{equation}
		The gauge group $\mathrm{G}$, which acts as usual via gauge transformation on the connection as $\omega \mapsto \tilde{\omega} = g\omega g^{-1} + dg \cdot g^{-1}$, acts from the left on $\Omega$ by $g \cdot \omega_\beta = (g^{-1})^*\omega_\beta$.
	\end{proposition}
	\begin{proof}
		This is standard, and is essentially just a rephrasing in the language of exterior differential systems of what it means to pullback the tangent bundle $T\Sigma$ to the complement of the zero section $M = \mathsf{Z}_B^c$. The statement about the left $\mathrm{G}$-action on $\Omega$ follows from the induced left, linear $\mathrm{G}$-action on $M$ on the fibre coordinates. It follows that there is a right $\mathrm{G}$-action on the horizontal distribution ~$\Delta$.
	\end{proof}
	Applying this well known situation to the case of a rank-3 vector bundle over a surface, we naturally obtain a horizontal 2-plane distribution on a 5-manifold.
	\begin{lemma}
		\label{lemma-horizontal_distribution_M5}
		Let $\pi : B \to \Sigma$ be a trivial rank-3 $\mathrm{G}$-bundle over a surface $\Sigma$, equipped with a compatible linear connection $\nabla$. Let $\mathsf{Z}_B \subset \mathrm{Tot}(B)$ be the zero section. Then the horizontal distribution $\Delta = \pi^*T\Sigma\vert_{M^5} \subset TM^5$ is a 2-plane distribution on the 5-manifold $M^5=\mathsf{Z}_B^c\subset\mathrm{Tot}(B)$ determined by the connection $\nabla$. For a generic connection $\nabla$, the derived flag of $\Delta$ is two-step bracket generating $\Delta \subset \Delta^\prime \subset \Delta^{(2)} = TM^5$, with growth vector $(2,3,5)$. The derived distributions $\Delta^\prime \subset \Delta^{(2)} \subseteq V$ are sub-bundles of the vertical bundle $V \subset TM^5$, consisting of vector fields tangent to the fibres of $\pi$, and are respectively determined by the curvature $\mathbf{R}$ of $\nabla$ and its covariant derivatives.
	\end{lemma}
	\begin{proof}
		This is the conditions of Proposition \ref{prop-total_space_EDS} with $n=2$ and $k=3$. Indeed, write let $e=(e_1,e_2,e_3)$ be a global frame of $B$ and local coordinates $(x^1,x^2)$ on $\Sigma$. Then the horizontal distribution $\Delta = \mathrm{span} \, \{X_1,X_2\}$ is given by Equation (\ref{eq-general_horizontal_distribution}). Write $\mathbf{R} = (R^\alpha_\beta) \in \underline{\mathrm{End}}(B) \otimes \Omega^2(\Sigma)$ as the matrix of curvature 2-forms, and $R^\alpha_{\beta ; k} = (\nabla_{\partial_{x^k}} R)^\alpha_\beta$ for $k=1,2$.
		The growth vector being $(2,3,5)$ for a generic connection $\nabla$ follows from Cartan \cite{Cartan_1910}, which is true for a generic 2-plane distribution on any 5-manifold. Then write $\Delta^\prime = \mathrm{span}\{X_1,X_2,X_3\}$ and $\Delta^{(2)} = \mathrm{span}\{X_1,X_2,X_3,X_4,X_5\}$ with $X_3 = [X_1,X_2]$ and $X_{k+3} = [X_k,X_3]$ for $k=1,2$. It is basic result in the theory of connections on a fibre bundle (e.g., expressing the curvature in terms of the Fr\"olicher-Nijenhuis bracket) that the curvature be vertical valued; we have by direct computation that 
		\begin{equation}
			\label{eq-vertical_curvature}
			X_3 = \sum_{\beta=1}^3 e_\alpha R^\alpha_\beta  \partial_{e_\beta} \, .
		\end{equation}
		Then the same is true of $\nabla \mathbf{R}$, and we have
		\begin{equation}
			\label{eq-vertical_covariant_derivatives}
			X_{k+3} = \sum_{\beta=1}^3 e_\alpha R^\alpha_{\beta ; k}  \partial_{e_\beta} \quad k=1,2 \, .
		\end{equation}
	\end{proof}
	We have the following result on the growth vector of $\Delta$, obtained by imposing restrictions on the curvature tensor.
	\begin{theorem}
		\label{thm-general_growth_vector}
		Consider the horizontal 2-plane distribution $\Delta \subset TM^5$ as in Lemma \ref{lemma-horizontal_distribution_M5}. Then
		\begin{enumerate}
			\item[\textit{(i)}] The horizontal distribution $\Delta = \Delta^\prime$ is integrable if and only if the curvature $\mathbf{R} = 0$ of $\nabla$ vanishes.
			\item[\textit{(ii)}] The growth vector of $\Delta$ is $(2,3)$ if and only if the curvature $\mathbf{R} \neq 0$ and is covariantly constant, $\nabla \mathbf{R} = 0$.
			\item[\textit{(iii)}] The growth vector of $\Delta$ is $(2,3,4)$ if and only if $\nabla \mathbf{R} \neq 0$ and there exists $X \in \underline{T\Sigma}$ such that $\nabla_X \mathbf{R} = 0$.
		\end{enumerate}
	\end{theorem}
	\begin{proof}
		The first two statements follow immediately from Lemma \ref{lemma-horizontal_distribution_M5}. For the final statement, suppose that the growth vector is $(2,3,4)$. Then the vectors $X_4,X_5 \in \Delta^{(2)}$ are linearly dependent. Since $\Delta^{(2)} \subset V$, they must be linearly dependent over $V$. But each vector has coefficients that are linear in the fibre variables from (\ref{eq-vertical_covariant_derivatives}), so it follows that they are linearly dependent if and only if there is a linear relation over $\mathcal{O}_{\Sigma}$ between $R^\alpha_{\beta ; 1}$ and $R^\alpha_{\beta ; 2}$ for local coordinates $(x^1,x^2)$ on $\Sigma$. By the tensorality of linear connections in tangential directions to $\Sigma$, this is equivalent to the existence of $X \in \underline{T\Sigma}$ such that $\nabla_X \mathbf{R} = 0$.
	\end{proof}
	It follows then that $\Delta$ admits an $2$-dimensional integral manifold on which through any point if and only if the curvature $\mathbf{R}=0$ of $\nabla$ vanishes; in this case, the integral manifolds are precisely the vertical translates of the zero section $Z_B$ embedded in $\mathrm{Tot}(B)$. There is then the question of how the gauge group interacts with the horizontal distribution. From Proposition \ref{prop-total_space_EDS}, there is a right $\mathrm{G}$-action on $\Delta$, and so it is natural to consider the distribution up to gauge equivalence. We have the following 
	\begin{corollary}
		\label{cor-growth_vector}
		The growth vector of $\Delta$ is gauge invariant.
	\end{corollary}
	\begin{proof}
		The conditions in Lemma \ref{lemma-horizontal_distribution_M5} and Theorem  \ref{thm-general_growth_vector} are all preserved under the action of the gauge group $\mathrm{G}$ on the bundle $B$.
	\end{proof}
	While there is this natural gauge symmetry, it is not in general true that the infinitesimal action of the gauge group acts as an infinitesimal symmetry of the distribution, since $\mathrm{G}$ acts naturally on the fibres of $\pi$ and hence on the vertical bundle $V \subset TM^5$. However, we may interpret the results here in terms of the classical Ambrose-Singer theorem \cite{MR63739}, \cite[Theorem 8.1]{MR0152974}. This theorem tells us that the $\mathfrak{g}$-valued curvature $\mathbf{R}$ determines the holonomy algebra of the relevant principle bundle equipped with connection, which interpreted here controls the first derived $\Delta^\prime$. Then the growth vector restrictions in Theorem \ref{thm-general_growth_vector}, seen as a failure of the derived flag of the horizontal distribution $\Delta$ to fill up the vertical distribution, shows that directions in which the curvature is covariantly constant fail to contribute directions on $\Sigma$ in which horizontal curves may generate elements of the holonomy algebra. 
	\subsection{Immersed surfaces in a 3-manifold}
	\label{ss-immersed_in_3-manifold}
	\par There is another direct analogue of the euclidean story in \S \ref{s-intro}. Let $N$ be a 3-manifold, equipped with a $\mathrm{G}$-structure on the tangent bundle $TN$ with compatible connection $\nabla^\circ$; a natural example is of course $\mathrm{G} = \mathrm{O}(3)$ if $N$ is a riemannian manifold with $\nabla^\circ$ the Levi-Civita connection. Suppose $\Sigma$ is a surface that $\psi : \Sigma \to N$ is an immersion. Then the bundle $B := \psi^*TN \to \Sigma$ is a rank-3 bundle that naturally splits as $B \cong T\Sigma \oplus N_{\Sigma}$, where $N_\Sigma$ is the normal bundle of $\psi(\Sigma) \subset N$. There is an induced connection $\nabla := \psi^*\nabla^\circ$, however, $\nabla$ may only be compatible with a subgroup $\mathrm{H} \subseteq \mathrm{G}$ (possibly trivial). Of course, one could also take other, higher rank $\mathrm{G}$-bundles on $N$ and pull them back to $\Sigma$ with $\psi$ as well, and look for circumstances in which the pullback bundle has a rank-3 summand that controls the geometry induced on $\Sigma$ by the immersion. Then, in the spirit of \S \ref{ss-horizontal}, one obtains a horizontal 2-plane distribution on complement of the zero section of the 5-dimensional total space of the relevant bundle that is determined by the determined by the geometry of the original bundle on $N$ and the immersion $\psi$.
	\par In the remainder of this article, we examine an important case, that of the projective 3-space $N = \mathbb{P}^3$, to which the question of the geometry of immersed surfaces in $\mathbb{P}^3$ has historically played an important role in differential geometry, differential equations, and algebraic geometry. This particular case blends many of the perspectives of the remarks above.
	\section{Projective surfaces \& rank four linear systems of PDE in the plane}
	\label{s-DG}
	\par In this section we review relevant features of the local projective geometry of a surface $\Sigma \subset \mathbb{P}^3$, which parallels the familiar euclidean story in \S \ref{s-intro}. This classical subject is morally equivalent to studying differential geometric features of rank-4 linear systems of PDE in the plane, following the work of Wilczynski \cite{MR0131232,MR1500783} and many of the other old masters, including Darboux, Bol, C\v ech, Cartan. A modern treatment can be found in Sasaki \cite{MR1165117,sasaki1999projective,MR2216951} and Sasaki \& Yoshida \cite{MR960834}. In essence, the surface $\Sigma$ is identified with the projective image of solutions to the rank-4 linear system, and conversely.  Such systems are governed by projective differential invariants determined by the conformal class of the second fundamental form coming from an immersion of a surface into projective 3-space $\mathbb{P}^3$, which are expressed naturally in a special choice of local coordinates known as asymptotic coordinates, the projective analogue of curvature lines on a euclidean surface.

	\subsection{Moving frame method for surfaces in \texorpdfstring{$\mathbb{P}^3$}{P3}}
	\label{ss-surfaces}
	\par A smooth surface $\Sigma$ is \emph{projective} if there is an immersion of $\Sigma$ into some projective space $\mathbb{P}^n$. Let $V^{n+1}$ be a vector space of dimension $n+1$ over $\mathbb{K}$. Recall that the projective space $\mathbb{P}^n$ is constructed as the space of lines in $V$, $\mathbb{P}^n = V - \{0\} / (v \sim k v)$, where $k \in \mathbb{K}^*$. The content of this article is on the local geometry of surfaces $\Sigma$ equipped with an immersion $\psi : \Sigma \to \mathbb{P}^3$, which parallels the classical study of the differential geometry of surfaces in euclidean 3-space. Then $\psi$ is locally a diffeomorphism, and we identify $\Sigma$ locally with its image $\psi(\Sigma) \subset \mathbb{P}^3$. Ultimately, we want to know the projective differential invariants of $\Sigma$ that are independent of the choice of $\psi$. Working in the analytic topology, we may freely shrink the image as necessary. Intrinsically, the existence of such an immersion can be formulated in the language of $\mathrm{G}$-structures \cite[Ch. IV, \S 4]{MR355886} with the group $\mathrm{G} = \mathrm{PSL}(4) = \mathrm{SL}(4) / \mathrm{center}$, and can be effectively described with the data of a projective moving frame on $\Sigma$. We follow the presentation of this material largely from Sasaki \cite[\S 1]{MR1165117} and \cite[Ch. 4, 5]{sasaki1999projective}. Much here can be generalized in a straightforward way to hypersurfaces in $\mathbb{P}^{n}$ for $n \geq 4$, though we do not require this generalization in the context of this article.
	\par The given immersion $\psi$ is naturally identified at each point $p \in \Sigma$ with a projective vector $\psi(p) = e_0(p) \in \mathbb{P}^3$. Up to scaling, $e_0(p)$ can also naturally be identified with its lift to an honest vector in $V^4 - \{0\}$. Working locally, we may choose lifts of tangent vector fields $e_1, e_2 \in \underline{T\Sigma}$ that are linearly independent at each $p \in \Sigma$. Then an additional vector field $e_3 \in V^4$ may be chosen independent of $e_0,e_1,e_2$ at each $p \in \Sigma$, so that the ordered set $e=(e_0,e_1,e_2,e_3)$ provides a local framing of a smooth rank-4 vector bundle $E  \to \Sigma$, called a projective moving frame along $\Sigma$. We will simply call $e$ a projective frame on $\Sigma$, which can naturally be regarded as a local section of the a left\footnote{Our convention as with Sasaki is that the elements of the frame $e$ are regarded as row vectors.} principal $\mathrm{GL}(4)$ bundle over $\Sigma$. It is natural to consider frames such that $\det(e)=1$, so that locally, a projective frame defines a map $e : \Sigma \to \mathrm{PSL}(4)$ that we identify with the same character as the frame $e$ by a misuse of notation. This gives a reduction of the structure group of $E$ from $\mathrm{GL}(4)$ to $\mathrm{PSL}(4)$, which can always be done locally. 
	\par Under these considerations, the projective frame $e$ satisfies the differential equation 
	\begin{equation}
		\label{eq-moving_frame}
		d e = \omega e \, ,
	\end{equation}
	where $\omega = (\omega^\alpha_\beta) \in \mathfrak{sl}(4) \otimes \Omega^1(\Sigma)$ is a Lie algebra valued 1-form on $\Sigma$ that encodes fundamental infinitesimal information of how $\psi$ immerses $\Sigma$ into $\mathbb{P}^3$. Equation (\ref{eq-moving_frame}) is known as the moving frame equation. Then we have that $d e_\beta = \omega^\alpha_\beta e_\alpha$ and $\omega$ is trace free, $\mathrm{tr}(\omega) = \omega^\alpha_\alpha =0$. To signify tangent directions with our indices, we will let greek indices range from $0,\dots,3$, and latin indices range from $1,2$. 
	\par By our construction $\omega$ is the pullback of the Maurer-Cartan form to $\Sigma$ by $e$, and naturally satisfies an integrability condition known as the Maurer-Cartan equation
	\begin{equation}
		\label{eq-mauer_cartan}
		d \omega - \omega \wedge \omega = 0 \, .
	\end{equation} 
	The frame $e$ is determined at each point $p \in \Sigma$ only up to left translation by a group element $g(p) \in \mathrm{PSL}(4)$, acting as $e \mapsto \tilde{e} = g e$. Here $g : \Sigma \to \mathrm{PSL}(4)$ is assumed to be a locally smooth mapping. A simple calculation shows the dependence of $\omega$ on the choice of frame is related by a gauge transformation, i.e., $\omega \mapsto \tilde{\omega}$ is given by
	\begin{equation}
		\label{eq-gauge_transformation}
		\tilde{\omega} = g\omega g^{-1} + dg \cdot g^{-1} \, .
	\end{equation}
	It is clear then that $\omega$ defines the data of a flat connection $\nabla$ on the rank-4 vector bundle $E \to \Sigma$ with connection 1-forms given by the $\omega^\alpha_\beta$ above and curvature tensor $\mathbf{R}^{\circ} = d \omega - \omega \wedge \omega = 0$ given by Equation (\ref{eq-mauer_cartan}); conversely, prescribing the data of a smooth rank-4 $\mathrm{PSL}(4)$ bundle $E \to \Sigma$ with flat connection $\nabla$ gives precisely the data of a smooth immersion $\psi : \Sigma \to \mathbb{P}^3$.
	\subsection{Projective differential invariants of surfaces}
	\label{ss-differential_invariants}
	\par We would like to know the differential invariants of an immersed surface $\psi(\Sigma) \equiv \Sigma \subset \mathbb{P}^3$ that determine the surface up to a projective motion, independent of the choice of immersion $\psi : \Sigma \to \mathbb{P}^3$. We say that a surface $\Sigma \subset \mathbb{P}^3$ is \emph{very general} if the pair $(\Sigma,\psi)$ is determined up to projective motion by the differential geometric quantities discussed later in this section and summarized in Proposition \ref{prop-canonical_invariants}. Strictly speaking, a part of what makes projective surfaces special is that these quantities are \emph{not} truly differential invariants in the usual sense, as there are ``general" surfaces that are \emph{not} very general. Such surfaces, called \emph{projectively applicable}, have projective deformations that preserve the ``differential invariants" given here. But a very general surface has no such deformations, and thus the quantities given in this section are actually differential invariants. This is discussed in more detail in \S \ref{sss-proj_applicable}. 
	\par To begin, we first note the following. By construction, there is a canonical splitting of the vector bundle $E$ as 
	\begin{equation}
		\label{eq-splitting}
		E \cong L \oplus T\mathbb{P}^3|_{\Sigma} \cong L \oplus T\Sigma \oplus L^* \, ,
	\end{equation}
	where $L \to \Sigma$ is the conormal bundle of $\Sigma \subset \mathbb{P}^3$  generated by the immersion $\psi = e_0$, and $L^* \to \Sigma$ is the dual line generated by the vector field $e_3$, canonically isomorphic to the normal bundle $L^* \cong  N_{\Sigma} = T\mathbb{P}^3|_{\Sigma}/ T\Sigma$. In the case that $\mathbb{K} = \mathbb{C}$ and $\Sigma$ is closed, it follows from the celebrated Chow theorem \cite{MR0033093} that $E$ is defined from intrinsic geometric data of $\Sigma$, i.e., it is built from the tangent bundle and divisors\footnote{The bundles in forming the summands of $E$ are generically not trivial, but working locally may we assume that they are all simultaneously trivialized.}. Then we have for some $m > 0$ that $\underline{L} \cong \mathcal{O}_{\Sigma}(-m)$, whence $\underline{L}^* \cong \mathcal{O}_{\Sigma}(m)$. 
	\par From the first row of the moving frame equation (\ref{eq-moving_frame}), we have that 
	\begin{equation*}
		de_0 = \omega^0_0e_0 + \omega^1_0 e_1 + \omega^2_0e_2 \, ,
	\end{equation*}
	so that $\omega^3_0 = 0$, which follows immediately by definition of the vector field $e_3$. Let $\omega^i := \omega^i_0$ be the natural coframe of the tangent fields $e_i$, $i =1,2$. It follows that these smooth 1-forms are independent, basic forms of the bundle $T^*\Sigma \subset E^*$. Then the Maurer-Cartan equation (\ref{eq-mauer_cartan}) gives by exterior differentiation of $\omega^3_0 = 0$ that 
	\begin{equation*}
		\omega^i \wedge \omega^3_i = 0 \, ,
	\end{equation*}
	where $i=1,2$ is summed over. Hence by the Cartan lemma, there are local smooth functions $h_{ij} = h_{ji}$ such that $\omega^3_i = h_{ij} \, \omega^j$. We are thus natually lead to consider the smooth symmetric tensor $h \in S^2(\Sigma) \otimes \underline{L}^*$ given by
	\begin{equation}
		\label{eq-second_fundamental_form}
		h = h_{ij} \, \omega^i \otimes \omega^j \, ,
	\end{equation}
	which is the projective analogue of the second fundamental form in euclidean geometry. We will always assume that $h$ is nondegenerate.
	\par It can be seen that the conformal class of $h$ is independent of the choice of frame for $E$. From the split form of the rank-4 vector bundle $E$ in Equation (\ref{eq-splitting}), it follows that an automorphism preserving this structure is given by
	\begin{equation}
		\label{eq-frame_automorphism}
		g=\left(\begin{array}{lll}
			\lambda & 0 & 0 \\
			\beta & a & 0 \\
			\rho & \gamma & \upsilon
		\end{array}\right) \in \mathrm{PSL}(4) \, .
	\end{equation}
	As $\det(g) = 1$, we thus have $\lambda, \upsilon \in \mathcal{O}^*_{\Sigma}$ and $a : \Sigma \to \mathrm{GL}(2)$. Moreover, $\gamma,\beta$ are valued in $\mathbb{\underline{K}}^2$ and $(\mathbb{\underline{K}}^2)^*$, respectively. Write $a=(a^i_j)$. Then the gauge transformed connection $\tilde{\omega}$ in Equation (\ref{eq-gauge_transformation}) has relevant components given by 
	\begin{equation}
		\label{eq-transformed_components}
		\tilde{\omega}^i = \lambda(a^{-1})^i_j \, \omega^j \, , \hspace{10mm} \tilde{\omega}^3_j = \upsilon^{-1} a^i_j \, \omega^3_i \, .
	\end{equation}
	Hence by applying the Cartan lemma in the new frame again it follows that $\tilde{h}_{ij} = \lambda \upsilon^{-1} a^k_i a^l_j h_{kl}$, or that $(\tilde{h}_{ij}) = \lambda \upsilon^{-1} a(h_{ij})a^T$. Thus we have 
	\begin{equation}
		\label{eq-conformal_II}
		\tilde{h} = \tilde{h}_{ij} \, \tilde{\omega}^i \otimes \tilde{\omega}^j = \lambda \upsilon^{-1} a^k_i a^l_j h_{kl}\, \omega^i \otimes \omega^j = \lambda \upsilon^{-1} \, h \, ,
	\end{equation}
	so that the conformal class $[h]$ is independent of the choice of frame. Let $\mathrm{P} \subseteq \mathrm{PSL}(4)$ be the parabolic subgroup generated by the linear automorphisms of $E$ in Equation (\ref{eq-gauge_transformation}). Then the canonical splitting of $E$ (\ref{eq-splitting}) gives a further reduction of the structure group from $\mathrm{PSL}(4)$ to $\mathrm{P}$; thus the flat connection $\nabla$ following (\ref{eq-gauge_transformation}) becomes a flat $\mathrm{P}$-connection on $E$. 
	\par The discussion here generalizes directly to hypersurfaces in $\mathbb{P}^{n}$, and shows that projective hypersurfaces carry a canonical (smooth) conformal structure, namely the conformal class of the projective second fundamental form. In the higher dimensional case $n \geq 4$, this conformal structure is determined by a normal Cartan connection \cite[Ch. IV, \S 4]{MR355886}, \cite[\S 2]{MR1165117}, which is integrable, or flat, if and only if $h$ is locally conformally flat. This occurs if and only if the hypersurface is contained in an open subset of a hyperquadric.
	\par In this paragraph we take $\mathbb{K} = \mathbb{C}$. Then the conformal structure described above is known as a \emph{holomorphic} conformal structure. As the splitting of $E$ in Equation (\ref{eq-splitting}) defines a canonical Hodge structure of weight two, holomorphic conformal structures have found natural applications in algebraic geometry, first utilized by Sasaki \& Yoshida \cite{MR996019}. Applications to moduli and variation of complex structure for lattice polarized K3 surfaces appear in Matsumoto, Sasaki, \& Yoshida \cite{MR1136204}, Doran \cite{MR1779161}, Doran, Doran \& Harder \cite{MR3822902}, the second author's dissertation \cite{schultz_geometry_2021}, and work by the second author with Malmendier \cite{HCSK324}. These cases generally deal with the higher dimensional case $n \geq 4$, and provide highly nontrivial examples of integrable holomorphic conformal structures. However, it should come as no surprise that the $n=3$ case of surfaces $\Sigma \subset \mathbb{P}^3$ have an induced conformal geometry that is quite distinct from the higher dimensional one.
	\par Let us continue in our study of the induced conformal geometry from a projective hypersurface immersion. We search for frames that simplify the geometric data as much as possible. From the relation of the conformal class $[h]$ to the projective frame in Equation (\ref{eq-conformal_II}) it follows that we may find a frame such that $|\det(h)|=1$. Any gauge transformation $g$ in Equation (\ref{eq-gauge_transformation}) stabilizing this condition will then satisfy $|\lambda \nu|=1=|\det(a)|$. It is trivial that the trace free condition is stabilized under gauge transformation, but we refine this condition by searching for a frame in which $\omega^0_0 + \omega^3_3 = 0$. Write $\beta,\gamma$ in Equation (\ref{eq-gauge_transformation}) as $\beta = (\beta^1,\beta^2)^T \in (\mathbb{\underline{K}}^2)^*$ and $\gamma = (\gamma_1,\gamma_2) \in \mathbb{\underline{K}}^2$. Under gauge transformation, $\omega^0_0 + \omega^3_3$ becomes
	\begin{equation}
		\label{eq-refined_trace}
		\tilde{\omega}_0^0+\tilde{\omega}_{3}^{3}=\omega_0^0+\omega_{3}^{3}+\nu^{-1}  \gamma^i \omega_i^{3}- \beta_i (a^{-1})_j^i \omega^j \, ,
	\end{equation}
	so that we may find a frame in which $\omega^0_0 + \omega^3_3 = 0$ as desired. 
	Now, taking the exterior derivative of $\omega^0_0 + \omega^3_3 = 0$ and utilizing the integrability condition (\ref{eq-mauer_cartan}) with the definition of the $h_{ij}$ yields 
	\begin{equation*}
		\left( h_{i j} \omega_{3}^j-\omega_i^0\right) \wedge \omega^i=0 .
	\end{equation*}
	Applying the Cartan lemma again yields local smooth functions $\ell_{ij}=\ell_{ji}$ such that 
	\begin{equation}
		\label{eq-ell_components}
		h_{i j} \omega_{3}^j-\omega_i^0 = \ell_{i j} \omega^j \, ,
	\end{equation}
	so that another symmetric tensor $\ell = \ell_{ij} \, \omega^i \otimes \omega^j$ may be obtained. However, in general $\ell$ may be degenerate (or even zero). In addition, it is useful to expand the coframe element $\omega^0_3$ as $\omega_{3}^0= - r_j \, \omega^j$, and define
	\begin{equation}
		\label{eq-r_eq}
		r = r_j \, \omega^j \, = - \omega_3^0 \, .
	\end{equation}
	\par Third order information about the immersion $\psi$ is determined as follows. We define $\Phi_{ijk}$ by the equation
	\begin{equation}
		\label{eq-cubic_components}
		\Phi_{ijk}\omega^k = d h_{ij} - h_{kj}\omega^k_j - h_{ik}\omega^k_j \, .
	\end{equation}
	Then $\Phi_{ijk}$ can be seen to be totally symmetric, and satisfies the apolarity condition $h^{ij}\Phi_{ijk}  = 0$ for $k=1,2$ since $|\det(h)|=1$. One then obtains a smooth cubic form as 
	\begin{equation}
		\label{eq-cubic_form}
		\Phi = \Phi_{ijk} \, \omega^i \otimes \omega^j \otimes \omega^k \, .
	\end{equation}
	The smooth cubic form $\Phi$ is known as the Wilczynski-Fubini-Pick form, or the Darboux cubic form, or simply for us the cubic form. An analogous computation as in Equation (\ref{eq-conformal_II}) shows that under a change of frame,
	\begin{equation}
		\label{eq-conformal_III}
		\tilde{\Phi} = \lambda \nu^{-1} \Phi \, .
	\end{equation}
	Then the significance of the quantities $\{\Phi,h,\ell,r\}$ that have been derived from the data of the immersion is given by the following result. 
	\begin{proposition}[Sasaki \cite{MR2216951}, pp. 118]
		\label{prop-projective_invariants}
		The conformal classes $[h]$, $[\Phi]$ of the projective second fundamental form and the Wilczynski-Fubini-Pick form are independent of the choice of projective moving frame for $\Sigma \subset \mathbb{P}^3$, and under a suitable gauge transformation stabilizing the conditions described above, the quantities $\ell$ and $r$ in Equations (\ref{eq-ell_components}, \ref{eq-r_eq}) transform with suitable contractions of $\Phi$. In particular, the quantities $\{\Phi,h,\ell,r\}$ are invariants of projectively equivalent surfaces, and in fact for very general surfaces determine the immersion up to projective automorphism.
	\end{proposition}
	The conscious reader will no doubt recognize that $\ell$ and $r$ may defined in an analogous way for hypersurfaces in $\mathbb{P}^{n}$ for $n \geq 4$. Yet in this case, $\ell$ and $r$ can be determined completely in terms of $h$ and $\Phi$, while this is false in our $n=3$ case \cite[\S 2.6]{MR2216951}. We discuss this in \S \ref{sss-proj_applicable}. Regardless, $h$ and $\Phi$ still play a significant role for surfaces, as we will see.
	\par To end this section, we note that a scalar $\mathcal{F}$ may be defined via contraction of $\Phi^{\otimes 2}$ with the inverse $h^{-1} = (h^{ij})$ as 
	\begin{equation}
		\label{eq-Fubini_invariant}
		\mathcal{F} = \Phi_{i j k} \Phi_{p q r} h^{i p} h^{j q} h^{k r}
	\end{equation}
	called the Fubini scalar. The obvious transformation properties of $\mathcal{F}$ under a change of frame imply that the form $\varphi = \mathcal{F}h$ is an absolute invariant.
	\begin{definition}
		\label{def-projective_metric}
		The symmetric tensor $\varphi = \mathcal{F}h$ is an absolute projective invariant called the \emph{projective metric}.
	\end{definition}
	%
	%
	%
	%
	An effective means of computing these quantities is provided by writing the moving frame equation (\ref{eq-moving_frame}) as a system of linear PDEs.
	\subsection{Rank-4 linear systems in the plane}
	\label{ss-rank-4_systems}
	\par  Near $p \in \Sigma$, let $(x,y)$ be local coordinates, and let $\psi : \Sigma \to \mathbb{P}^3$ be an immersion given locally by
	\begin{equation*}
		(x,y) \mapsto z(x,y) = \left[ \, z_0(x,y) \, : \, z_1(x,y) \, : \, z_2(x,y) \, : \, z_3(x,y) \, \right] \, \in \mathbb{P}^3 \, .
	\end{equation*}
	The key insight of Wilczynksi \cite{MR1500783}, and more recently by Sasaki \cite{sasaki1999projective,MR2216951}, Sasaki \& Yoshida \cite{MR960834} (see also Ferapontov \cite{MR1762804}), is that the immersion $\psi$ is completely determined by the existence of a rank four system of linear PDEs
	\begin{equation}
		\label{eq-rank_four_system}
		\scalemath{.9}{
			\begin{aligned}
				& \frac{\partial^2 z}{\partial x^2}= l \frac{\partial^2 z}{\partial x \partial y}+\alpha \frac{\partial z}{\partial x}+b \frac{\partial z}{\partial y}+\mu z, \\
				& \frac{\partial^2 z}{\partial y^2}= m \frac{\partial^2 z}{\partial x \partial y}+c \frac{\partial z}{\partial x}+\delta \frac{\partial z}{\partial y}+\nu z,
		\end{aligned}}
	\end{equation}
	to which the component functions $z_k(x,y)$, $k = 0,\dots,3$, of the projective vector $z(x,y) \in \mathbb{P}^3$ are linearly independent solutions of Equation (\ref{eq-rank_four_system}) near $p \in \Sigma$. Here all coefficient functions in the system are locally smooth. In turn, such a system is completely determined by local differential geometric data (including the data of the immersion), as was elucidated by Sasaki \& Yoshida  \cite{MR960834}. It is thus of interest to examine the transformation properties of the coefficient functions; Wilczynski \cite{MR1500783} determined that the most general form of transformation of the independent and dependent variables in the linear system (\ref{eq-rank_four_system}) that preserves its form is given by $(x,y;z) \mapsto (u,v;f(u,v)z)$, where $x = x(u,v), y= y(u,v)$ defines a local diffeomorphism $\phi :~ \Sigma \to ~\Sigma$ and $f \in \mathcal{O}^*_\Sigma$. A straightforward computation \cite[\S 3]{MR960834} reveals that under such a transformation, the coefficient functions $l,m$ transform as components of an element of the sheaf $S^2(\Sigma)  \otimes \underline{L}^*$, suggesting they are closely related to the conformal structure described in \S \ref{ss-differential_invariants}. 
	\par In particular, the assumption that the image of $\psi$ in $\mathbb{P}^3$ constitutes a surface implies that $e = (\psi, \psi_x, \psi_y, \psi_{xy})$ constitutes a frame of the rank-4 bundle $E \to \Sigma$. Then, writing (\ref{eq-rank_four_system}) in Pfaffian form $de = \omega e$ is nothing but equipping $E$ with a flat connection $\nabla$, whose connection 1-forms are precisely $\omega$.
	An elementary computation yields
	\begin{equation}
		\label{eq-connection_rank_4_sys}
		\omega=\left(\begin{array}{cccc}
			0 & d x & d y & 0 \\
			\mu \, d x & \alpha \, d x & b \, d x & l \, d x+d y \\
			\nu \, d y & c \, d y & \delta \, d y & d x+ m \, d y \\
			* & * & * & *
		\end{array}\right) \; .
	\end{equation}
	The integrability condition (\ref{eq-mauer_cartan}) is necessary and sufficient for the system to be of rank-4, which is necessary and sufficient \cite{MR1500783} for the image of $\psi$ to be a surface $\Sigma \subset \mathbb{P}^3$ in general position. Hence, the data of the rank-4 linear system (\ref{eq-rank_four_system}) determines the data $(\Sigma,\psi)$ up to projective automorphism, or equivalently, $(E,\nabla)$ up to gauge equivalence. In this view, the connection $\nabla$ is more fundamental than the linear system, since the latter determines only a choice of gauge. 
	\par It follows from (\ref{eq-second_fundamental_form}) that the conformal class $[h]$ that determines the projective metric is generated by the normal bundle valued covariant tensor 
	\begin{equation}
		\label{eq-proj_metric_coefficients}
		h = (l \, dx \otimes dx + dx \otimes dy + dy \otimes dx + m \, dy \otimes dy) \otimes \sigma \, ,
	\end{equation}
	where $\sigma \in \underline{L}^*$ is a local nonvanishing section of the normal bundle. Working in locally, we assume that $\underline{L}^* \cong \mathcal{O}_\Sigma$ is trivial. Then the choice of section $\sigma$ simply determines the conformal class $[h]$. The requirement that $h$ be nondegenerate thus means that $l m - 1 \neq 0$. We have the following elementary 
	\begin{proposition}
		\label{prop-proj_gauge_transformation}
		Let $f \in \mathcal{O}^*_\Sigma$. Then under the covariant transformation $z \mapsto ~fz$, the coefficients $l,m$ of $z_{xy}$ in the canonical system (\ref{eq-canonical_sys}) are preserved. Thus the covariant transformations of Wilczynski are contained in the isometry subgroup $\mathrm{Isom}(h) \subset~ \mathrm{P}$ of the tensor $h$.
	\end{proposition}
	\begin{proof}
		A simple computation with the rank-4 system (\ref{eq-rank_four_system}) yields the claim, on the other hand, every such covariant transformation $z \mapsto fz$ yields a corresponding gauge transformation $g_f \in \mathrm{P}$ of the connection form $\omega$ in (\ref{eq-connection_rank_4_sys}) that is easily computable. Computing $\tilde{h}$ from the gauge transformed connection $\tilde{\omega}$ directly shows $\tilde{h} = h$, thus the covariant transformations of Wilczynski act isometrically on $h$. 
	\end{proof}
	The remaining tensorial invariants described in \S \ref{ss-differential_invariants} can all be computed in a straightforward manner, but their expressions can be simplified by working in a preferred system of coordinates called \emph{asymptotic} coordinates. The coordinate lines of such a system are the projective version of lines of curvature in classical euclidean geometry.
	\begin{definition}
		\label{def-asymptotic}
		Let $p \in \Sigma \subset \mathbb{P}^3$. A tangent vector $Y \in T_p \Sigma$ is an \emph{asymptotic direction} at $p$ if $h_p(Y,Y) = 0$. A differentiable curve $C : t \mapsto C(t) \in \Sigma$ is an \emph{asymptotic curve} if the tangent vector field $C'(t)$ is everywhere an asymptotic direction in each $T_{C(t)}\Sigma$. A system of local coordinates $(x,y)$ on $\Sigma$ is \emph{asymptotic} if the coordinate curves $x =~ \mathrm{const.}$, $y = \mathrm{const.}$ are all asymptotic curves.
	\end{definition}
	It is clear that the notions in Definition \ref{def-asymptotic} are all independent of the conformal class $[h]$. Over $\mathbb{C}$, such a system of asymptotic coordinates exists near any $p \in \Sigma$; if one works over $\mathbb{R}$, it is necessary to take $h$ to have signature $(1,1)$. Thus in the latter case, we have the projective analogue of a hyperbolic surface in euclidean space. There are two important consequences of a surface being endowed with asymptotic coordinates $(x,y)$: first, that the asymptotic coordinate lines locally foliate the surface, and hence there is a local splitting of the tangent bundle $T\Sigma = T_x \oplus T_y$ splits locally into the tangent bundles of the foliating curves; second, that locally in asymptotic coordinates the conformal class $[h]$ is generated by the constant tensor
	\begin{equation}
		\label{eq-const_form}
		h = dx \otimes dy + dy \otimes dx \, .
	\end{equation}
	We shall henceforth assume that $\Sigma$ has been equipped with asymptotic coordinates.
	\par If $(x,y)$ are asymptotic coordinates near $p \in \Sigma$, the linear system becomes 
	\begin{equation}
		\label{eq-pre_canonical_system}
		\begin{aligned}
			& \frac{\partial^2 z}{\partial x^2}= \alpha \frac{\partial z}{\partial x}+b \frac{\partial z}{\partial y}+\mu z,  \\
			& \frac{\partial^2 z}{\partial y^2}= c \frac{\partial z}{\partial x}+\delta \frac{\partial z}{\partial y}+\nu z,
		\end{aligned}
	\end{equation}
	and the connection form $\omega$ reduces similarly to 
	\begin{equation}
		\label{eq-connection_pre_canonical_sys}
		\omega=\left(\begin{array}{cccc}
			0 & d x & d y & 0 \\
			\mu \, d x & \alpha \, d x & b \, d x & d y \\
			\nu \, d y & c \, d y & \delta \, d y & d x \\
			* & * & * & \alpha \, dx + \delta \, dy
		\end{array}\right) \; .
	\end{equation}
	Generally the frame $e$ defining (\ref{eq-connection_rank_4_sys}) does not yield an $\mathfrak{sl}(4)$-valued 1-form, and so a reduction of the structure group can be computed by utilizing the so-called normalization factor $\exp(2\theta)$ of the linear system \cite[\S 2]{MR960834}, \cite[\S 2.2]{MR2216951}, defined as
	\begin{equation}
		\label{eq-normalization_factor}
		\exp(2\theta) := \det(e) \in K^{-1}_\Sigma \,.
	\end{equation}
	Then $\theta = \theta(x,y)$ is a locally well-defined function. The classical Jacobi formula for differentiating determinants, when applied to the Maurer-Cartan form $\omega$ implies that 
	\begin{equation*}
		2 \, d\theta = \mathrm{tr}(\omega) \, ,
	\end{equation*}
	so that $\mathrm{tr}(\omega)$ can locally be integrated to find the normalization factor. Applying this to the connection form (\ref{eq-connection_pre_canonical_sys}), we find that 
	\begin{equation}
		\label{eq-norm_asymptotic}
		d \theta = \alpha \, dx + \delta \, dy \, .
	\end{equation}
	Thus, we may write $\alpha = \theta_x$ and $\delta = \theta_y$ in linear system (\ref{eq-pre_canonical_system}).
	\par With normalization factor in hand, this linear system is simplified under the covariant transformation $z \mapsto \exp(-\theta/2)z$ into a system of the form
	\begin{equation}
		\label{eq-canonical_sys0}
		\begin{aligned}
			& \frac{\partial^2 z}{\partial x^2}= b \frac{\partial z}{\partial y}+\widetilde{\mu} z,  \\
			& \frac{\partial^2 z}{\partial y^2}= c \frac{\partial z}{\partial x}+\widetilde{\nu} z,
		\end{aligned}
	\end{equation}
	where $b,c$ are preserved, and $\widetilde{\mu}, \, \widetilde{\nu}$ are respectively differential polynomials in $\mu,b,\theta$ and $\nu,c,\theta$ that can be computed in a straightforward manner. Such a form is called the \emph{canonical} form of the linear system (\ref{eq-rank_four_system}); we shall drop the decorated symbols and focus on linear systems in canonical form
	\begin{equation}
		\label{eq-canonical_sys}
		\begin{aligned}
			& \frac{\partial^2 z}{\partial x^2}= b \frac{\partial z}{\partial y}+\mu z,  \\
			& \frac{\partial^2 z}{\partial y^2}= c \frac{\partial z}{\partial x}+\nu z.
		\end{aligned} 
	\end{equation}
	The connection form then simplifies to 
	\begin{equation}
		\scalemath{.85}{
			\omega =\left(\begin{array}{cccc}
				\label{eq-connection_canonical_sys}
				0 & d x & d y & 0 \\
				\mu \, d x & 0 & b \, d x & d y \\
				\nu \, d y & c \, d y & 0 & d x \\
				\left(b \nu+ \mu_y\right) d x+\left(\mu c+ \nu_x\right) d y & b c \, d x+\left( c_x+\nu\right) d y & \left(\mu+ b_y\right) d x+b c \, d y & 0
			\end{array}\right)} .
	\end{equation}
	The canonical system and its associated connection form will be of fundamental importance in our analysis. It can be shown that under an arbitrary reparameterization preserving the asymptotic foliations of $\Sigma$ that the coefficients $b,c$ in Equation (\ref{eq-canonical_sys}) transform as tensor densities of weight $(-1,2)$ and $(2,-1)$ respectively \cite[(5.8)]{Ovsienko_Tabachnikov_2004}, \cite[(2.6)]{MR1762804}, where the ordered grading comes from the splitting of the tangent bundle. We have the following
	\begin{proposition}
		\label{prop-canonical_invariants}
		The fundamental projective invariants (\ref{prop-projective_invariants}), (\ref{def-projective_metric}) of a surface $\Sigma \subset \mathbb{P}^3$ are given in asymptotic coordinates from the canonical system as follows: the projective metric and the cubic form are
		\begin{equation}
			\label{eq-canonical_invariants}
			\begin{aligned}
				\varphi &= 8  bc (dx \otimes dy + dy \otimes dx) \, , \\
				\Phi &= -2  b  dx^{\otimes 3} -2 c  dy^{\otimes 3} \, .
			\end{aligned}
		\end{equation}
		The remaining invariants $\ell,r$ are simple differential polynomials in $b,c,\mu,\nu$, and can be found in \cite[Equation (2.15)]{MR2216951}.
	\end{proposition}
	\begin{proof}
		The computation of $\Phi$ comes directly from the definition (\ref{eq-cubic_components}) by utilizing the connection form (\ref{eq-connection_canonical_sys}). Then the projective metric $\varphi = \mathcal{F}h$ comes by the Fubini scalar (\ref{eq-Fubini_invariant}) computed from $\Phi$ and the constant tensor $h$ in (\ref{eq-const_form}).
	\end{proof}
	\begin{lemma}
		\label{lemma-integrability}
		A system of linear PDEs of the form (\ref{eq-canonical_sys}) defines a surface $\Sigma \subset \mathbb{P}^3$ equipped with asymptotic local coordinates $(x,y)$ if and only if the following conditions hold:
		\begin{equation}
			\label{eq-integrability_conditions}
			\begin{aligned}
				-2  b_y  \nu- b \nu_y- \mu_{yy}+\mu_x  c+2 \mu  c_x+ \nu_{xx} &= 0 \; , \\
				-2 b_y  c-b c_y + c_{xx}+2 \nu_x &= 0 \;, \\
				-2\mu_y-b_{yy}+ b_x  c+2  b c_x &=0 \; . 
			\end{aligned}
		\end{equation}
	\end{lemma}
	\begin{proof}
		The necessity is nothing but the integrability condition $\mathbf{R} =0$ applied to the connection form of the canonical system (\ref{eq-connection_canonical_sys}). The sufficiency is guaranteed by noting that the rank of the canonical system is four if and only if (\ref{eq-integrability_conditions}) holds, and that these conditions are equivalent to the compatibility conditions from the equality of mixed partials $(z_{xx})_{yy} = (z_{yy})_{xx}$.
	\end{proof}
	In this way, the integrability conditions should be viewed as the projective version of the Gauss-Codazzi equations (\ref{eq-codazzi}) in classical differential geometry of surfaces in euclidean space. We remark trivially that such relations are satisfied for \emph{any} $\Sigma \subset \mathbb{P}^3$; that is, we learn nothing of the local geometry from these equations being satisfied. In \S \ref{s-2_plane_distr}, we will give another explanation of the origin of these equations as a consequence of the differential Bianchi identity applied to a natural connection on a rank-3 sub-bundle of $E$.
	\subsection{Projective geometry associated to the vanishing of invariants}
	\label{ss-geometry_proj_surf}
	It follows from Proposition \ref{prop-projective_invariants} that very generally, a projective surface $\Sigma \subset \mathbb{P}^3$ is determined up to automorphism by the projective metric $\varphi$ and cubic form $\Phi$. In this case, we have $bc \neq 0$, and $b,c$ do not depend on any auxiliary parameters. Inherently, the canonical system (\ref{eq-canonical_sys}) is sensitive to the choice of immersion; this data is reflected in $\mu,\nu$, as will be explained below. This distinction becomes crucial when considering global aspects of the immersion $\psi$, particularly in the complex analytic category where it is common to immerse surfaces in projective space by Fuchsian systems of PDEs. In this case, fundamental solutions of the canonical system are multivalued, and it is of great importance to understand both the local and global aspects of the associated monodromy groups and their representation on fibres of the bundle $E$. See for example Sasaki \& Yoshida \cite{MR996019}, Hara, Sasaki, \& Yoshida \cite{MR1040172}, Matsumoto, Sasaki \& Yoshida \cite{MR1136204}, Clingher, Doran, \& Malmendier \cite{MR3767270}, Malmendier \& the second author \cite{MR4494119}, etc.
	\par Even in the case that one is only concerned with the geometry of the surface $\Sigma$ itself, there is still the chance that one may deform the surface $\Sigma$ while keeping the projective metric $\varphi$ and the cubic $\Phi$ the same. Such surfaces are called \emph{projectively applicable}, and were studied for example by Cartan \cite{Cartan_1920}. This geometric behavior is unique in dimension two. We will give more detail in \S \ref{sss-proj_applicable}.
	\par There is a very rough stratification of projective surfaces based on the relative vanishing of the invariants $\varphi,\Phi$, the proofs of which and references can be found in many places. See for example \cite{MR2216951} and references therein. 
	\subsubsection{Quadric surfaces}
	\label{sss-quadrics}
	Quadric surfaces in $\mathbb{P}^3$ are characterized by the vanishing of the cubic $\Phi$. In fact, a similar statement holds for hypersurfaces in $\mathbb{P}^{n}$; see \cite[Theorem 2.3]{MR2216951} for a simple proof. It follows that quadric surfaces in $\mathbb{P}^3$ are the $2$-dimensional projective analogue of conformally flat spaces in higher dimension \cite{MR31310}. Thus $b=c=0$, and so we also have that the projective metric $\varphi = 0$ vanishes. 
	\par The integrability conditions (\ref{eq-integrability_conditions}) under the quadric condition $b=c=0$ simplify greatly to $\mu_y = \nu_x = 0$; thus, the canonical system decouples into a system of second order ODEs:
	\begin{equation}
		\label{eq-quadric_sys}
		\begin{aligned}
			& \frac{\partial^2 z}{\partial x^2}= \mu(x) z,  \\
			& \frac{\partial^2 z}{\partial y^2}= \nu(y) z.
		\end{aligned}
	\end{equation}
	Then the connection form (\ref{eq-connection_canonical_sys}) decomposes by inspection into the (outer) tensor product of the connection form of the respective ODEs. Over $\mathbb{C}$, when the ODEs are Fuchsian, an analogous statement holds for the associated monodromy representations, and the classical Gauss-Schwarz theory expresses $\mu,\nu$ in terms of Schwarzian derivatives. Such decouplings play an important role in applications to subtle algebro-geometric aspects of lattice polarized K3 surfaces of Picard rank-18 in the references above preceding this section.
	\par By a suitable rescaling of $z$, we may assume $\mu = \nu = 0$. This simplifies the equations dramatically, $z_{xx} = 0 = z_{yy}$, which is trivially solved. The independent solutions $z_0 = xy, \, z_1=x, \, z_2=y, \, z_3=1$ satisfy the quadratic equation $z_0z_3 = z_1z_2$, which for $\mathbb{K} = \mathbb{R}$ defines the unique quadric surface of signature $(1,1)$ in $\mathbb{P}^3$ up to projective automorphism. For $\mathbb{K} = \mathbb{C}$, all quadrics are projectively equivalent to $z_0z_3 = z_1z_2$. We again stress that such a rescaling is particularly violent at the level of the differential systems, since for example in the complex analytic category it discards \emph{all} local and global data about the associated (Fuchsian) system. This discussion reflects the fact that quadric surfaces in $\mathbb{P}^3$ are doubly ruled surfaces. More generally, we have the following.
	\subsubsection{Ruled Surfaces}
	\label{sss-ruled_surfaces}
	Recall that a projective surface $\Sigma$ is \emph{ruled} if it contains a 1-parameter family of projective lines\footnote{When we say that $\Sigma$ is \emph{ruled} we mean that it is singly ruled, ruling out the possibility that $\Sigma$ be a quadric.}. Ruled surfaces play a very important role in the theory of surfaces in $\mathbb{P}^3$ via constructions known as line congruences \cite{MR2216951}. In practice, such a surface can be specified by a pair of curves $s(x),t(x)$ in $\mathbb{P}^3$ with common local parameter $x$ and line joining the respective points on each curve parameterized by $y$, so that an immersion $\psi : \Sigma \to \mathbb{P}^3$ is given by
	\begin{equation}
		\label{eq-ruled_surface_immersion}
		\psi(x,y) = s(x) + yt(x) \, \in \mathbb{P}^3 \, .
	\end{equation}
	\par It is straightforward to show that in this situation, the parameters $(x,y)$ give asymptotic coordinates on $\Sigma$. From Lemma \ref{lemma-integrability}, it follows that each component $z_k(x,y)$ of $\psi$ is a solution of the canonical system (\ref{eq-canonical_sys}) of $\Sigma$; insisting that each component be linear in $y$ simply means that $c=0=\nu$. Clearly, one could swap the roles of $x$ and $y$ in the argument above and get an analogous statement. Thus, a ruled surface $\Sigma$ has vanishing projective metric $\varphi = 0$ and cubic form $\Phi  \neq 0$. By Proposition \ref{prop-projective_invariants} it follows this condition in fact characterizes ruled surfaces.
	\par The integrability conditions (\ref{eq-integrability_conditions}) simplify in this case to $\mu_{yy} = 0$, $b_{yy}+2\mu_y = 0$. These equations can easily be integrated to find that the canonical system of a ruled surface is of the form
	\begin{equation}
		\label{eq-canonical_sys_ruled}
		\frac{\partial^2 z}{\partial x^2}= (\alpha(x)y^2+\beta(x)y+\gamma(x))\frac{\partial z}{\partial y} + (-\alpha(x)y+\delta(x)) z  \, , \hspace{3mm}
		\frac{\partial^2 z}{\partial y^2} = 0 \, .
	\end{equation}
	Evaluating (\ref{eq-canonical_sys_ruled}) in the form of (\ref{eq-ruled_surface_immersion}) shows that the curves $u(x),v(x)$ satisfy the second order system of coupled linear ODEs
	\begin{equation}
		\label{eq-ruled_surface_ODE_sys}
		\begin{aligned}
			& \frac{d^2 s}{d x^2} = \delta(x)u(x) + \gamma(x)v(x)  \, ,\\
			& \frac{d^2 t}{d x^2} = -\alpha(x)u(x) + (\beta(x)+\delta(x))v(x) \, .
		\end{aligned} 
	\end{equation}
	Ruled surfaces have their geometry determined by the curves $s(x),t(x)$, and the dependencies in the ODEs above. Writing $S(x) = (s(x),t(x))^T$ and $\cdot = d/dx$,  (\ref{eq-ruled_surface_ODE_sys}) becomes a matrix equation $\ddot S = QS$, where $Q = Q dx^{\otimes 2}$ is regarded as a matrix valued quadratic differential. Wilczynski showed \cite{MR0131232} that ruled surfaces have differential invariants generated in terms of invariant data of $Q dx^{\otimes 2}$, such as the determinant and trace. See also Sasaki \cite[\S 1.2]{MR2216951}.
	\subsubsection{Projectively applicable surfaces}
	\label{sss-proj_applicable}
	Examining the integrability conditions in Lemma \ref{lemma-integrability} reveals that $b,c$ do not determine $\mu,\nu$ uniquely, and so there is a chance to deform $\Sigma$ without changing $\varphi$ and $\Phi$. Such surfaces are called projectively applicable \cite{Cartan_1920}, and play an interesting role in the geometry of projective surfaces. In this section we follow the presentation of Sasaki \cite[\S 2.6]{MR2216951}.
	\begin{definition}
		\label{def-proj_deformation}
		A morphism $\mathscr{F} : \Sigma \to \Sigma^\prime$ of projective surfaces is said to be a \emph{projective deformation} of $\Sigma$ into $\Sigma^\prime$ if for all $p \in \Sigma$, there exists a projective transformation $g_p \in \mathrm{Aut}(\mathbb{P}^3) \cong \mathrm{PSL}(4)$ such that $g_p(p) = \mathscr{F}(p)$ and that for any smooth curve $C \subset \Sigma$ passing through $p$, the curve $g_p(C)$ has second order contact at $\mathscr{F}(p) \in \Sigma^\prime$ with the image curve $\mathscr{F}(C) \subset \Sigma^\prime$.
	\end{definition}
	\begin{definition}
		\label{def-proj_applicable}
		A projective surface $\Sigma \subset \mathbb{P}^3$ is \emph{projectively applicable} if it possesses a nontrivial projective deformation.
	\end{definition}
	\par The significance of projective applicability is the following
	\begin{proposition}
		\label{prop-proj_applicable}
		Suppose that $\mathscr{F} : \Sigma \to \Sigma^\prime$ is a projective deformation, and suppose that $(x,y)$ and $(x^\prime,y^\prime)$ are asymptotic local coordinates on $\Sigma$ and $\Sigma^\prime$, respectively. Then there is a canonical splitting $\mathscr{F}^*T\Sigma^\prime = \mathscr{F}^*(T_{x^\prime} \oplus T_{y^\prime}) = T_{x} \oplus T_{y}$, i.e., a projective deformation carries asymptotic directions to asymptotic directions. Moreover, in the respective canonical systems (\ref{eq-canonical_sys}) for $\Sigma,\Sigma^\prime$, we have $b = b^\prime$ and $c = c^\prime$ relative to the respective asymptotic coordinates. Consequently, $\mathscr{F}^*\varphi^\prime = \varphi$ and $\mathscr{F}^*\Phi^\prime = \Phi$. The conditions above are also sufficient for the existence of a projective deformation.
	\end{proposition}
	\begin{proof}
		That $\mathscr{F}$ is a projective deformation immediately implies the statement about preserving asymptotic directions by the necessary second order contactedness. Then locally we may regard $(x,y) = (x^\prime,y^\prime)$, so that $\Sigma,\Sigma^\prime$ share a common asymptotic parameterization. It follows that up to a projective transformation, we have $z^\prime = fz$ for some nonvanishing $f$. Proposition \ref{prop-proj_gauge_transformation} shows that such a covariant transformation acts isometrically on the constant tensor $h$ in (\ref{eq-const_form}). Then an argument similar to that result shows that $b = b^\prime$ and $c = c^\prime$. Hence $\mathscr{F}^*\varphi^\prime = \varphi$ and $\mathscr{F}^*\Phi^\prime = \Phi$. The sufficiency follows by reversing the argument and noting that the derivatives of the scaling factor $f$ are not constrained at the point of contact. 
	\end{proof}
	\par Elementary examples of projectively applicable surfaces are quadrics and ruled surfaces, due to the arbitrary functions $\mu(x),\nu(y),\gamma(x)$ that appear in the $z$-coefficients of (\ref{eq-quadric_sys}) and (\ref{eq-canonical_sys_ruled}) respectively. These surfaces have vanishing projective metric $\varphi = 0$. There are many other examples of generically nonruled projectively applicable surfaces that are derived from multiparameter families of bivariate hypergeometric functions, such as Appell's $F_2$ and $F_4$ functions. See for example \cite[Example 2.20]{MR2216951} and \cite{MR1858701}.
	\par In this generic case $bc \neq 0$, a straightforward argument shows that the coefficients $\mu,\nu$ in the canonical system (\ref{eq-canonical_sys}) may depend on at most three arbitrary constants \cite[Equation (2.37)]{MR2216951}. Such a situation is very special and reflects the projective geometry of $\Sigma$ as follows.
	\begin{theorem}[Proposition 2.17, \cite{MR2216951}]
		\label{thm-gaussian_curvature}
		Let $\Sigma \subset \mathbb{P}^3$ be a generic surface equipped with asymptotic coordinates $(x,y)$ and canonical system (\ref{eq-canonical_sys}) with $bc \neq 0$, admitting three independent projective deformations, that is, $\mu,\nu$ depend on three arbitrary parameters. Then the Gaussian curvature of the projective metric $\varphi$ is constant.
	\end{theorem}
	\begin{proof}
		Let $K = K(\varphi)$ be the Gaussian curvature of the projective metric $\varphi$ of $\Sigma$. This is easily computable as 
		\begin{equation}
			K = \frac{-1}{8bc} \frac{\partial^2}{\partial x \partial y} \log(bc) \, .
		\end{equation}
		In this case it follows by utilizing the integrability conditions in Lemma \ref{lemma-integrability} that the three arbitrary parameters in $\mu, \nu$ lead to the following quantities being identically zero:
		\begin{equation*}
			\frac{\partial^2}{\partial x \partial y} \log\left(\frac{b}{c}\right) = 0 \, , \quad \frac{\partial}{\partial x} \left( \frac{1}{bc} \frac{\partial^2}{\partial x \partial y} \log(c) \right) = 0 \,, \quad \frac{\partial}{\partial y} \left( \frac{1}{bc} \frac{\partial^2}{\partial x \partial y} \log(b) \right) = 0\, .
		\end{equation*}
		These three equations together imply that $K \equiv \mathrm{const}$.
	\end{proof}
	\section{Derived reduction of the canonical system}
	\label{s-derived_reduction}
	Using the language of exterior differential systems (again see \cite{BCG3} and \cite{Ivey_Landsberg2003}, or \cite{Stormark_2000}), we may locally encode every system of PDE as an exterior differential system (EDS), or dually as a vector field distribution, as in \S \ref{s-rank_3}. In this section, we study the canonical system in Equation (\ref{eq-canonical_sys}) from the EDS perspective. While this system is naturally associated to a flat connection on a rank-4 bundle $\pi : E \to \Sigma$ as in \S \ref{ss-rank-4_systems}, we show that there is a canonical way to reduce this data to the rank-3 sub-bundle $B := E/L^*$, with an induced connection $\bar{\nabla}$ that both completely determines and is completely determined by the projective geometry of $\Sigma \subset \mathbb{P}^3$.
	\subsection{A rank-3 EDS determined by a projective surface}
	\label{ss-rank-3_pfaffian}
	\par Consider the projective surface $\Sigma \subset \mathbb{P}^3$, equipped with asymptotic local coordinates $(x,y)$. Let $J^2=J^2(\Sigma, \mathbb{K})$ denote the space of (local) 2-jets from $\Sigma$ to $\mathbb{K}$. This is a geometric structure that encodes equivalence classes of local sections $z : \Sigma \to L$ up to second order derivatives, from these considerations we assume $L \cong \mathcal{O}_{\Sigma}$. The bundle $J^2$ forms a smooth 8-dimensional manifold that fibres over $\Sigma$, equipped with local coordinates $(x,y,z,z_{x}, z_{y}, z_{xx}, z_{xy}, z_{yy})$ and a canonical projection $\pi_2 : J^2 \to \Sigma$. There is an analogous 5-dimensional first order jet space $J^1 \subset J^2$, equipped with a compatible projection $\pi_1 : J^1 \to \Sigma$. 
	\par On $J^2$, we may define the contact one-forms
	\begin{equation}
		\begin{cases}
			\theta_z := dz - z_{x} \, dx - z_{y} \, dy \, , \\
			\theta_x := dz_{x} - \, z_{xx} \, dx - z_{xy} \, dy \, , \\
			\theta_y := dz_{y} - \, z_{xy} \, dx - z_{yy} \, dy \, .
		\end{cases}
	\end{equation}
	These one-forms generate a differential ideal, $\mathrm{Ct} = \langle \theta_z, \, \theta_x, \, \theta_y \rangle_\mathrm{diff}$, called the \emph{canonical contact system} on $J^2$. Since the rank-4 linear system (\ref{eq-rank_four_system}), and hence the canonical system (\ref{eq-canonical_sys}), are of finite type (all third order derivatives of $z$ can be expressed in terms of $J^2$ coordinates), it is sufficient to only consider second order jet space $J^2$ and the ideal ~$\mathrm{Ct}$.
	We thus obtain from the canonical system an embedded submanifold $M^6 = \{(x,y,z,p, q, s) \; \vert \; (z,p,q,s) \neq (0,0,0,0)\} \subset J^2$ from the locus of the linear equations
	\begin{equation}
		\label{eq-M6_eqs}
		\begin{aligned}
			r &= bp + \mu z \, , \\
			t &= cq + \nu z \, , 
		\end{aligned}
	\end{equation}
	where we have set $p = z_x$, $q = z_y$, $r = z_{xx}$, $t=z_{yy}$, and $s = z_{xy}$. From this perspective it is clear that $M^6 = \mathsf{Z}_E^c \subset \mathrm{Tot}(E)$ is the complement of the zero section $\mathsf{Z}_E$ in the total space of the rank-4 vector bundle $E$ in (\ref{eq-splitting}), using notation analogous to Proposition \ref{prop-total_space_EDS}. Let $\iota: M^6 \to J^2$ be the associated inclusion map. Then the restriction $\Omega := \iota^\ast (\mathrm{Ct})$ defines a rank-3 EDS (or Pfaffian system) on $M^6$; explicitly, $\Omega = \langle \omega_z, \, \omega_x, \, \omega_y \rangle_\mathrm{diff}$ where $\omega_\alpha := \iota^*\theta_{\alpha}$, $\alpha=z,x,y$, whence
	\begin{equation}
		\label{eq-M6_rank3_EDS}
		\begin{cases}
			\omega_z = dz - p  dx - q  dy, \\
			\omega_x = dp - ( b q + \mu z )  dx - s  dy, \\ 
			\omega_y = dq -  s  dx - ( c p + \nu z )  dy.
		\end{cases}
	\end{equation}
	\par Moreover, the differential ideal $\Omega$ is nothing but the ideal generated by the first three rows of the connection $\nabla$ pulled back to $M^6$ under the bundle projection map $\pi : E \to \Sigma$. This is completely sufficient for our purposes, since the final row of the connection matrix is of the form
	\begin{equation}
		\label{eq-total_derivative}
		\omega_s = ds - D_y (r) dx - D_x (t) dy \, ,
	\end{equation}
	where $D_x$ and $D_y$ are the well known total derivative operators, and $r,t$ are given in (\ref{eq-M6_eqs}). This is in essence a middle ground between prolongation of the canonical system (\ref{eq-canonical_sys}) and the fact that the linear system is of finite type and rank-4. 
	\par If instead one considers the rank-4 Pfaffian system $\tilde{\Omega} := \langle \pi^*\omega\vert_{M^6} \rangle_{\mathrm{diff}}$ on $M^6$ analogously generated as in \S \ref{ss-horizontal}, then it follows from those results that if $\tilde{\Delta} = \tilde{\Omega}^\perp$, then $\tilde{\Delta}$ is an integrable 2-plane distribution on $M^6$, as in \S \ref{s-rank_3}. This is a fundamental feature of EDS perspective for solutions of PDEs \cite{Stormark_2000}, allowing the study the rank-4 system \eqref{eq-canonical_sys} from yet another differential geometric perspective, in the sense that each solution of the system \eqref{eq-canonical_sys} defines a 2-dimensional integral manifold $\psi: \Sigma \to M^6$, which satisfies by definition $\psi^*\tilde{\Omega}=0$, and moreover  that $\psi^\ast(dx \wedge dy) \neq 0$; conversely, every 2-dimensional integral submanifold $\psi(\Sigma) \subset M^6$ for the EDS $\tilde{\Omega}$ which $\psi^\ast (dx \wedge dy) \neq 0$ defines a solution to \eqref{eq-canonical_sys}. This is in perfect agreement with the discussion above and in \S \ref{ss-horizontal}, since the curvature $\mathbf{R}=0$ of the connection $\nabla$ vanishes canonically. Yet this perspective reveals nothing about the local geometry of $\Sigma$ discussed in \S \ref{ss-differential_invariants}. In the sequel, we show that the rank-3 EDS $\Omega$ can be used to detect more subtle geometric aspects of $\Sigma$, and can be reduced to yield a certain horizontal 2-plane distribution $\bar{\Delta}$ on a natural ~$M^5 \subset M^6$. 
	\par The Pfaffian system $\Omega$ is the annihilator of a rank-3 vector field distribution $\Delta$ on $M^6$ given by
	\begin{equation}
		\label{eq:canonical_Delta}
		\Delta = \mathrm{span}\begin{cases}
			X = \partial_x + p \, \partial_z + (b q + \mu z) \, \partial_{p} + s \, \partial_{q} , \\
			Y = \partial_y + q \, \partial_z + s \, \partial_{p} + (c p + \nu z) \, \partial_{q}, \\ 	
			Z = \partial_{s}. \\
		\end{cases}
	\end{equation}
	The vector field $Z$, tangent to the fibre direction $L^*$ on $M^6=\mathrm{Tot}(E)$, plays special importance. It is easy to check that the distribution spanned by \eqref{eq:canonical_Delta} generically has growth vector $(3,5,6)$. In exactly the same manner discussed in Proposition \ref{prop-total_space_EDS} and Corollary \ref{cor-growth_vector}, the parabolic gauge group $\mathrm{P} \subset \mathrm{PSL}(4)$ of the connection $\nabla$ can be checked to preserve the growth vector. In fact, since $\mathrm{P}$ is parabolic and preserves the splitting of $E = L \oplus T\Sigma \oplus L^*$ in (\ref{eq-splitting}), for $g \in \mathrm{P}$ the action $g \cdot Z$ preserves the tangency to $L^*$.
	We have the following result summarizing this section.
	\begin{lemma}
		\label{lemma-rank3_M6}
		The canonical system (\ref{eq-canonical_sys}) of an immersed surface $\psi : \Sigma \to ~\mathbb{P}^3$ is  determined by the rank-3 Pfaffian system $\Omega$ in (\ref{eq-M6_rank3_EDS}) on the 6-manifold $M^6$. This EDS is the annihilator of a smooth 3-plane distribution $\Delta \subset TM^6$ which is always bracket generating and has growth vector $(3,5,6)$. The parabolic gauge group $\mathrm{P}$ of the connection $\nabla$ on $E$ acts naturally on $\Omega$ and $\Delta$ and preserves the growth vector; moreover, for all $g \in \mathrm{P}$, the vector $g \cdot Z$ is always tangent to $L^*$.
	\end{lemma}
	\begin{proof}
		The only thing necessary to check is the statement about the growth vector, which in principle could change with either the vanishing of the projective metric $\varphi$ or the cubic form $\Phi$ from Proposition \ref{prop-canonical_invariants}, or by incorporating the integrability conditions (\ref{eq-integrability_conditions}). However, implementing the integrability conditions in the derived flag, and checking each case in \S \ref{ss-geometry_proj_surf}, including when the surface is projectively applicable, the growth vector is always checked to be $(3,5,6)$. 
	\end{proof}
	Thus we may regard Lemma \ref{lemma-rank3_M6} as the capturing the geometry of the pair $(\Sigma,\psi)$ in the language of exterior differential systems. That the growth vector is always $(3,5,6)$ is a manifestation of Lemma \ref{lemma-integrability}, which says that the integrability conditions in Equation (\ref{eq-integrability_conditions}) are the only obstructions to the linear system (\ref{eq-canonical_sys}) defining a projective surface. We then have the following statement.
	\begin{proposition}
		\label{prop-rank3_M6-proj_applicable}
		Let $\Sigma \subset \mathbb{P}^3$ with associated rank-3 Pfaffian system $\Omega$ on $M^6$. Then we may consider the equivalence class of 3-plane distributions on $M^6$ that correspond to all gauge equivalent representations of $\Delta$ parabolic group $\mathrm{P} \subset \mathrm{PSL}(4)$. If $\Sigma$ is very general, then the distribution $\Delta$ is completely determined by the canonical system (\ref{eq-canonical_sys}). If $\Sigma$ is projectively applicable as in \ref{ss-geometry_proj_surf}, then the associated distribution is determined by the canonical system, parameterized by the relevant projective deformations of $\Sigma$.
	\end{proposition}
	We will see the significance of the vector field $Z$ in the following.
	\subsection{Derived Cauchy reduction}
	\label{ss-cauchy}
	\par A vector field $V \in \underline{\Delta}$ is said to be a \emph{Cauchy characteristic} if $[V, \Delta] \subseteq \Delta$, that is, the Lie derivative of any vector field in $\Delta$ with respect to $V$ lies in $\Delta$. The utility of Cauchy characteristics are as follows. A distribution admitting Cauchy characteristics can always be reduced to a lower rank distribution on a  lower-dimensional space, the latter given by the level-sets of the associated invariant functions \cite[\S 3.4]{Stormark_2000}. A direct computation shows the following 
	\begin{lemma}
		\label{lemma-cauchy_characteristic}
		The 3-plane distribution $\Delta$ admits no Cauchy characteristics. The vector field $Z$ is the only Cauchy characteristic of the derived distribution $\Delta'$. The same is true for all gauge equivalent 3-plane distributions under the $\mathrm{P}$-action on $\Delta$.
	\end{lemma}  
	\par Although the distribution \eqref{eq:canonical_Delta} admits no Cauchy characteristics,  and therefore $\Delta$ cannot be reduced in this way, we may still utilize the Cauchy characteristic of $\Delta^\prime$ to reduce the 3-plane distribution $\Delta$ as follows.
	\par Let $M_{s_0}^5 \subset M^6$ be the manifold defined by the level sets $s = s_0$ of the Cauchy characteristic vector field $Z = \partial_{s}$ of the first derived $\Delta^\prime$, parameterized by constant local sections $s=s_0$ of the (trivialized) normal bundle $\underline{L}^* \cong \mathcal{O}_{\Sigma}$. For each such $s_0$, we may define the vector fields
	\begin{equation}
		\label{eq-level_set_vector_fields}
		\begin{aligned}
			X_{s_0} &= \partial_x + p  \partial_z + (b q + \mu z) \partial_{p} + s_0  \partial_{q} , \\
			Y_{s_0} &= \partial_y + q  \partial_z + s_0  \partial_{p} + (c p + \nu z)  \partial_{q}, \\
		\end{aligned}
	\end{equation}
	on $M^6$. The derived Cauchy characteristic $Z$ is transverse to the submanifold $M^5_{s_0}$. These vector fields generate a one-parameter family of 2-plane distributions $\Delta_{s_0} = \mathrm{span}\{ X_{s_0}, \, Y_{s_0}\}$ which are invariant with respect to $Z$, and hence are well-defined on $M^5_{s_0}$ for all $s_0$. That is, for each fixed value of $s_0$, the 3-plane distribution $\hat{\Delta}_{s_0} := \mathrm{span} \, \{X_{s_0},Y_{s_0},Z\} \subset TM^6$ admits $Z$ as a Cauchy characteristic, so we may define the Cauchy reduced rank-2 distribution $\bar\Delta_{s_0} := \bar{\pi}_\ast(\Delta_{s_0})$ on $M^5_{s_0}$ where $\bar{\pi}: M^6 \to M^5_{s_0}$ is the standard projection map to the level set. We may write $\bar{\Delta}_{s_0} = \mathrm{span} \, \{\bar{X}_{s_0},\bar{Y}_{s_0}\}$, where the additional decoration indicates that the vectors live on $M^5_{s_0}$, but are defined by the relevant quantities in (\ref{eq-level_set_vector_fields}). We check that this non-traditional Cauchy reduction is well-defined with the following 
	\begin{lemma}
		\label{lemma-cauchy_reduction}
		For all constant local sections $s_0 \in \underline{L}^*$, the distribution $\hat{\Delta}_{s_0}$ defines a 3-plane distribution on $M^6$, with growth vector generically $(3,4,6)$. The vector field $Z \in \underline{\hat{\Delta}}_{s_0}$ is a Cauchy characteristic. The Cauchy reduced distribution $\bar{\Delta}_{s_0}$ defines a generic bracket generating 2-plane distribution on the 5-dimensional submanifold $M^5_{s_0} \subset M^6$, so that the growth vector of $\bar{\Delta}_{s_0}$ is $(2,3,5)$ without imposing additional constraints on the geometry of $\Sigma$. The growth vector restrictions on $\bar{\Delta}_{s_0}$ in the sufficient directions of the statements in Theorem \ref{thm-general_growth_vector} are independent of the choice of $s_0$. 
	\end{lemma}
	\begin{proof}
		From Lemma \ref{lemma-rank3_M6}, the growth vector of the 3-plane distribution $\Delta$ on $M^6$ is always $(3,5,6)$. A routine calculation shows that the closely related distribution  $\hat{\Delta}_{s_0}$, parameterized by constant local sections $s_0 \in \underline{L}^*$, has growth vector generically $(3,4,6)$. It is simple to see that, on $M^6$, $\left[Z,X_{s_0}\right] = 0 = \left[Z,Y_{s_0}\right]$, so that $Z$ is a Cauchy characteristic of $\hat{\Delta}_{s_0}$. Then by Cauchy reduction \cite[\S 3.4]{Stormark_2000}, we obtain the 2-plane distribution $\bar{\Delta}_{s_0}$ on $M^5_{s_0}$ defined above. As in Lemma \ref{lemma-horizontal_distribution_M5}, the growth vector of the 2-plane distribution is generically $(2,3,5)$. The possible growth vector restrictions, as seen in the statements of Theorem \ref{thm-general_growth_vector}, are: $(2)$, if integrable; $(2,3)$, and $(2,3,4)$. Since $\bar{X}_{s_0},\bar{Y}_{s_0}$ depend linearly on $s_0$, each of these possibilities can be checked to be independent of $s_0$.
	\end{proof}
	\par We call the process in Lemma \ref{lemma-cauchy_reduction} \emph{derived} Cauchy reduction. Per this result, it suffices to consider the case where $s_0 = 0$.  There are several advantages to considering the zero section of $L^*$ in $M^6$. We have the following results that summarize this discussion and follow from Lemmas \ref{lemma-cauchy_characteristic} and \ref{lemma-cauchy_reduction}.
	\begin{theorem}
		\label{thm-5-manifold}
		The manifold $M^5 := M^5_{s_0=0}$, which by definition is the restriction of the zero section of the line sub-bundle $L^* \subset E$ to $M^6$, is canonically diffeomorphic to both $Z_B^c \subset \mathrm{Tot}(B)$, the complement of the zero section $Z_B$ in the total space $\mathrm{Tot}(B)$ of the rank-3 vector bundle $B = E / L^*$ over $\Sigma$, and the analogous submanifold of the first order jet space $J^1$ of $\Sigma$. Then the process of derived reduction as in Lemma \ref{lemma-cauchy_reduction} and evaluating at the zero section $s_0=0$ equips $M^5$ with a $2$-plane distribution $\bar{\Delta} := \bar{\Delta}_{s_0=0}$ whose geometry is determined by the projective invariants of $\Sigma$.
	\end{theorem}
	\par Finally we may connect this discussion back to the ideas in \S \ref{s-rank_3}, by showing that $\bar{\Delta}$ is the horizontal distribution on $M^5=\mathsf{Z}_B^c\subset\mathrm{Tot}(B)$ of a $\mathrm{G}$-connection $\bar{\nabla}$ on $B$, where $\mathrm{G} = \mathbb{K}^* \times \mathrm{SO}(1,1) \subset \mathrm{P}$ is the largest subgroup that both stabilizes the direct sum structure $B = L \oplus T\Sigma \subset E$ of the zero section defining $B$, and acts conformally on the constant tensor $h$ in Equation (\ref{eq-const_form}). Here $\mathrm{G}$ is embedded diagonally in $\mathrm{P}$ as $g = \mathrm{diag}(\lambda,\exp(\alpha),\exp(-\alpha),\lambda^{-1})$, for $\lambda \in \mathbb{K}^*$ and $\alpha \in \mathbb{K}$. Notice that $B$ is again trivial and $\underline{B} = \mathrm{span}\,\{z,p,q\}$.
	\begin{corollary}
		\label{cor-connection_M5}
		The 2-plane distribution $\bar{\Delta}$ on $M^5=\mathsf{Z}_B^c$ is the horizontal distribution of a $\mathrm{G} = \mathbb{K}^* \times \mathrm{SO}(1,1)$ connection $\bar{\nabla}$ the rank-3 bundle $B$, whose matrix $\bar{\omega}$ of connection 1-forms is given by
		\begin{equation}
			\label{eq-connection_B}
			\bar{\omega} =\left(\begin{array}{ccc}
				0 & d x & d y \\
				\mu \, d x & 0 & b \, d x  \\
				\nu \, d y & c \, d y & 0 \\
			\end{array}\right) \,.
		\end{equation}
	\end{corollary}
	\begin{proof}
		Let $\iota_0 : M^5 \hookrightarrow M^6$ be the inclusion map induced by the zero section of $L^*$. Then the EDS $\bar{\Omega} := \iota_0^*\Omega$ is a rank-3 Pfaffian system on $M^5$, generated by the 1-forms
		\begin{equation}
			\label{eq-reduced_connection-1-forms}
			\begin{aligned}
				\bar{\omega}_z &= dz - pdx - qdy, \, \\
				\bar{\omega}_x &= dp - (bq + \mu z)dx, \, \\
				\bar{\omega}_y &= dq - (cp + \nu z)dy \, .
			\end{aligned}
		\end{equation}
		These 1-forms define the desired connection $\bar{\nabla}$ on $B$ as in Proposition \ref{prop-total_space_EDS}, yielding the matrix $\bar{\omega}$ of connection 1-forms in (\ref{eq-connection_B}). It is a routine manner to check that $\bar{\nabla}$ is a $\mathrm{G} = \mathbb{K}^* \times \mathrm{SO}(1,1)$ connection. Then, by Lemma \ref{lemma-horizontal_distribution_M5}, $\bar{\Omega} = \mathrm{ann}(\bar{\Delta})$ is the annihilator of the 2-plane distribution $\underline{\bar{\Delta}} = \mathrm{span} \, \{ \bar{X},\bar{Y}\}$ from Theorem \ref{thm-5-manifold} with
		\begin{equation}
			\label{eq-horizontal_proj_vector_fields}
			\begin{aligned}
				\bar{X} &= \partial_x + p  \partial_z + (b q + \mu z) \partial_{p} \, , \\
				\bar{Y} &= \partial_y + q  \partial_z  + (c p + \nu z)  \partial_{q} \, . \\
			\end{aligned}
		\end{equation}
		Thus $\bar{\Delta}$ is horizontal for $\bar{\nabla}$.
	\end{proof}
	\par It is very important to note that, as opposed to the situation in Lemma \ref{lemma-rank3_M6}, the growth vector of $\hat{\Delta}_{s_0}$ is only \emph{generically} $(3,4,6)$, and the growth vector \emph{is} sensitive to the vanishing of projective invariants as in \S \ref{ss-geometry_proj_surf}. Hence a similar statement is true for the generically $(2,3,5)$-distribution $\bar{\Delta}$ on $M^5$ in Theorem \ref{thm-5-manifold}, which shows that more refined features of the projective geometry of $\Sigma$ show up in this setting. These relationships will be the theme of \S \ref{s-2_plane_distr}. 
	\par It also very important to note that the group $G$ acts \emph{only} on the invariants $\varphi$ and $\Phi$ for $\Sigma$, which are \emph{independent} of any immersion $\psi : \Sigma \to \mathbb{P}^3$. This does not affect a very general surface, but all projectively applicable surfaces have immersions impacted by the presence of projective deformations; e.g., the discussion on quadric surfaces in \S \ref{sss-quadrics}. In the next section, we wish to study only how the projective geometry of $\Sigma \subset \mathbb{P}^3$ influences the 2-plane distribution $\bar{\Delta}$ on $M^5$ and conversely, and \emph{not} of the pair $(\Sigma,\psi)$.
	\section{Horizontal 2-plane distributions from projective surfaces}
	\label{s-2_plane_distr}
	Having reduced the problem studying the local geometry of a projective surface to a suitable $\mathrm{G}$-connection on a rank-3 bundle over the surface, we study here the confluence of results from \S \ref{s-rank_3} and \S \ref{s-derived_reduction} for projective surfaces $\Sigma \subset \mathbb{P}^3$, independent of the data of any immersion. We again emphasize that the integrability condition $\mathbf{R} = 0$ of the connection $\nabla$ on the rank-4 bundle $E \to \Sigma$ is satisfied for \emph{any} surface in $\mathbb{P}^3$ per Lemma \ref{lemma-integrability}. However, the connection $\bar{\nabla}$ on the rank-3 bundle $B \to \Sigma$ detects more refined aspects of the projective geometry of $\Sigma$. 
	\par An elementary computation reveals that the curvature tensor $\bar{\mathbf{R}} = d\bar{\omega} - \bar{\omega} \wedge \bar{\omega}$ of $\bar{\nabla}$ is given by 
	\begin{equation}
		\label{eq-curvature_B}
		\bar{\mathbf{R}} =\left(\begin{array}{ccc}
			0 & -(b\nu + \mu_y)dx \wedge dy & (c\mu + \nu_x) dx \wedge dy \\
			0 & -bc \, dx \wedge dy & (c_x + \nu)dx \wedge dy  \\
			0 & -(b_y+\mu)dx \wedge  dy & bc \, dx \wedge dy \\
		\end{array}\right) \,.
	\end{equation}
	Our first utility of the curvature $\bar{\mathbf{R}}$ of $\bar{\nabla}$ is the following observation. Upon inspection one sees that the last row of the connection form $\omega$ for $\nabla$ on the rank-4 bundle $E \to \Sigma$, encoded by the 1-form $\omega_s$ in (\ref{eq-total_derivative}) and whose coefficients are expressed in terms of total derivatives, is simply obtained from suitable contractions of the components of $\bar{\mathbf{R}}$ with the coordinate tangent vectors $\partial_x,\partial_y$. In fact, we have
	\begin{theorem}
		\label{thm-bianchi}
		The integrability conditions $\mathbf{R} = 0$ of the canonical system (\ref{eq-canonical_sys}) are a consequence of the differential Bianchi identity $d_{\bar{\nabla}}\bar{\mathbf{R}} = 0$.
	\end{theorem}
	\begin{proof}
		A direct computation of the necessary $d_{\bar{\nabla}}\bar{\mathbf{R}} = 0$ recovers the integrability equations in Lemma \ref{lemma-integrability}.
	\end{proof}
	We now study the associated projective geometry of $\Sigma$ according to the curvature $\bar{\mathbf{R}}$ and covariant derivatives with respect to $\bar{\nabla}$, in the spirit of \S \ref{ss-horizontal}.
	\begin{lemma}
		\label{lemma-vanishing_curvature}
		The curvature tensor $\bar{\mathbf{R}} = 0$ if and only if the surface is a ruled surface whose ruling curve lies in the intersection of two linear complexes. In particular, $\bar{\mathbf{R}} = 0$ for a quadric surface. 
	\end{lemma}
	\begin{proof}
		It is clear from (\ref{eq-curvature_B}) that if $\bar{\mathbf{R}} = 0$, we must have $bc = 0$; hence the projective metric $\varphi = 0$ vanishes and so $\Sigma$ is ruled. For a quadric, we have $b=0=c$ and $\mu = 0 = \nu$, so that $\bar{\mathbf{R}} = 0$. Assume without loss of generality that $c = 0$ and $b \neq 0$. As in \S \ref{sss-ruled_surfaces}, we may further assume that $\nu = 0$. Then the resulting equations from $\bar{\mathbf{R}} = 0$ are $\mu_y=0$ and $b_y+\mu = 0$, which are trivially solved as 
		\begin{equation}
			\label{eq-vanishing_curvature_ruled_surface}
			b = \beta(x)y + \gamma(x) \, , \quad \mu = - \beta(x) \, .
		\end{equation}
		Hence as in Equation (\ref{eq-canonical_sys_ruled}), the canonical system of such a surface is given by
		\begin{equation}
			\label{eq-ruled_surface_can_sys_R=0}
			\frac{\partial^2 z}{\partial x^2}= (\beta(x)y+\gamma(x))\frac{\partial z}{\partial y} -\beta(x) z  \, , \hspace{3mm}
			\frac{\partial^2 z}{\partial y^2} = 0 \, , 
		\end{equation}
		and the ruling curves $s(x),t(x)$ from $\psi(x,y) = s(x)+yt(x)$ satisfy (\ref{eq-ruled_surface_ODE_sys}), which here become
		\begin{equation}
			\label{eq-ruled_surface_ODE_sys_R=0}
			\begin{aligned}
				& \frac{d^2 s}{d x^2} = -\beta(x)s(x) + \gamma(x)t(x)  \, ,\\
				& \frac{d^2 t}{d x^2} = 0 \, .
			\end{aligned} 
		\end{equation}
		In fact, due to the relevant transformation properties of (\ref{eq-ruled_surface_ODE_sys}) (see Sasaki \cite[Theorem 1.8]{MR2216951}), we may consider $\beta=0$. Moreover, we have $\gamma = 0$ if and only if $\Sigma$ is a quadric. Such doubly ruled surfaces lie in the intersection of two linear complexes. In case $\gamma \neq 0$, the curve $s(x)$ can be shown to satisfy a \emph{reducible} fourth order ODE with vanishing first Wilczynski invariant \cite[\S 1.1]{MR2216951}, and hence lying in the intersection of two linear complexes.
	\end{proof}
	\begin{lemma}
		\label{lemma-ruled_surface_curvature}
		If $\Sigma$ is a general ruled surface as in \S \ref{sss-ruled_surfaces}, then $\bar{\mathbf{R}} \neq 0$, and $\bar{\nabla}_{\partial_y} \bar{\mathbf{R}} = 0$, where $\partial_y$ is tangent to the ruling lines on $\Sigma$.
	\end{lemma}
	\begin{proof}
		It follows from Equations  (\ref{eq-canonical_sys_ruled}) and (\ref{eq-curvature_B}) that for a general ruled surface, $\bar{\mathbf{R}} \neq 0$. Then a direct computation shows that  $\bar{\nabla}_{\partial_y} \bar{\mathbf{R}} = 0$.
	\end{proof}
	\par We may now characterize the horizontal 2-plane distribution $\bar{\Delta}$ on $M^5=\mathsf{Z}_B^c$ in terms of the projective geometry of $\Sigma \subset \mathbb{P}^3$ determined by the projective metric $\varphi$ and the cubic form $\Phi$.
	\begin{theorem}
		\label{thm-proj_surface_growth_vector}
		Let $\Sigma \subset \mathbb{P}^3$ be a projective surface and $\bar{\Delta}$ be the horizontal 2-plane distribution on $M^5=\mathsf{Z}_B^c$ of the connection $\bar{\nabla}$.
		\begin{enumerate}
			\item[(\textit{i})] $\bar{\Delta}$ is integrable if and only if $\Sigma$ is a ruled surface whose ruling lies in the intersection of two linear complexes. It is sufficient that $\Sigma$ be a quadric.
			\item[(\textit{ii})] There are no projective surfaces with growth vector $(2,3)$.
			\item[(\textit{iii})] $\bar{\Delta}$ has growth vector $(2,3,4)$ if and only if there exists $X \in \underline{T\Sigma}$ such that $\bar{\nabla}_X \bar{\mathbf{R}} = 0$. It is sufficient that $\Sigma$ be ruled.
			\item[(\textit{iv})] $\bar{\Delta}$ has growth vector $(2,3,5)$ if and only if $\bar{\mathbf{R}} \neq 0$ and is not covariantly constant in any local vector fields on $\Sigma$. It is sufficient that $\Sigma$ be very general.
		\end{enumerate}
	\end{theorem}
	\begin{proof} Each point follows from results in this section and \S \ref{ss-horizontal}.
		\begin{enumerate}
			\item[(\textit{i})] This follows from Theorem \ref{thm-general_growth_vector} (\textit{i}) and Lemma \ref{lemma-vanishing_curvature}.
			\item[(\textit{ii})] A direct computation shows $\bar{\nabla} \bar{\mathbf{R}} = 0$ if and only if $\bar{\mathbf{R}} = 0$, so by Theorem \ref{thm-general_growth_vector} (\textit{ii}), there are no projective surfaces with growth vector $(2,3)$.
			\item[(\textit{iii})] This follows from Theorem \ref{thm-general_growth_vector} (\textit{iii}) and Lemma \ref{lemma-ruled_surface_curvature}.
			\item[(\textit{iv})] This follows from Lemma \ref{lemma-horizontal_distribution_M5} and the definition of $\bar{\nabla}$.
		\end{enumerate}
	\end{proof}
	It is natural to ask if there are any non-ruled surfaces that satisfy (\textit{iii}) in Theorem \ref{thm-proj_surface_growth_vector}. We have computed as  an example an interesting 3-parameter family of generically non-ruled projectively applicable surfaces which determines a horizontal 2-plane distribution $\bar{\Delta}$ with growth vector $(2,3,4)$. 
	\begin{proposition}
		\label{prop-3_param_family}
		Consider the linear system
		\begin{equation}
			\label{eq-234_example}
			\begin{aligned}
				\frac{\partial^2 z}{\partial x^2}& = \frac{4k_1\left(k_1 y+k_2\right) }{\left(4 x+k_3\right)^2} \frac{\partial z}{\partial y} \, , \\
				\frac{\partial^2 z}{\partial y^2}& = \frac{4 x+k_3}{\left(k_1 y+k_2\right)^2} \frac{\partial z}{\partial x} \, , \\
			\end{aligned}
		\end{equation}
		with $k_1,k_2,k_3 \in \mathbb{K}$. Then (\ref{eq-234_example}) is of rank-4, and hence the solutions define a 3-parameter family of projectively applicable surfaces $\Sigma = \Sigma_{k_1,k_2,k_3} \subset \mathbb{P}^3$ with $(x,y)$ as asymptotic coordinates.
	\end{proposition}
	\begin{proof}
		One checks that the integrability conditions (\ref{eq-integrability_conditions}) hold, and so by Lemma \ref{lemma-integrability}, (\ref{eq-234_example}) is of rank-4, and is hence the the canonical system of a 3-parameter family $\Sigma \subset \mathbb{P}^3$ of projective surfaces. This family is projectively applicable by the results in \S \ref{sss-proj_applicable}.
	\end{proof}
	It follows that the coefficients in (\ref{eq-234_example}) determine the projective metric $\varphi$ and cubic form $\Phi$ as 
	\begin{equation}
		\label{eq-234_example_invariants}
		\begin{aligned}
			\varphi &= \frac{4k_1}{(4x+k_3)(k_1y+k_2)}(dx \otimes dy + dy \otimes dx) \, , \\
			\Phi &= \frac{-8k_1\left(k_1 y+k_2\right) }{\left(4 x+k_3\right)^2} dx^{\otimes 3} +  \frac{-2(4 x+k_3)}{\left(k_1 y+k_2\right)^2} dy^{\otimes 3} \, .
		\end{aligned}
	\end{equation}
	From these expressions, we see that $\Sigma$ is ruled if and only if $k_1=0$, and there are no quadric surfaces in the family. Moreover, the Gaussian curvature $K(\varphi)$ is easily computed to be $K(\varphi) \equiv 0$, which agrees with Theorem \ref{thm-gaussian_curvature} necessitating the constancy of $K(\varphi)$ for a 3-parameter family of projectively applicable surfaces. Moreover, it is easy to see that the curvature $\bar{\mathbf{R}} \neq 0$, and also that a direct computation shows $\bar{\nabla}_X\bar{\mathbf{R}} = 0$, where $X \in \underline{T\Sigma}$ is the vector field
	\begin{equation}
		\label{eq-234_example_vector}
		X = \left(4 x+k_3\right) \partial_x+ k_1\left(k_1 y+k_2\right) \partial_y \, .
	\end{equation}
	It follows that the horizontal 2-plane distribution $\bar{\Delta}$ has growth vector $(2,3,4)$ by Theorem \ref{thm-proj_surface_growth_vector} (\textit{iii}), though of course this can also be checked directly.
	\par It is of significant interest to determine precise geometric properties of the general surfaces in Theorem \ref{thm-proj_surface_growth_vector} (\textit{iv}), whose bracket generating $(2,3,5)$-distribution $\bar{\Delta}$ has enhanced symmetry \cite[Ch. 17]{Stormark_2000}; in particular, those with maximal 14-dimensional symmetry  given by the simple Lie algebra $\mathfrak{g}_2$. This is the subject of ongoing work by the authors.
	\bibliographystyle{abbrv}
	\bibliography{Ashley_Schultz_BIB}
\end{document}